\documentclass[11pt]{amsart}
\usepackage[margin=1in]{geometry}
\usepackage{amssymb,amsfonts,amsmath,mathtools,amsthm}
\usepackage{enumerate}
\usepackage{mathrsfs}
\usepackage[unicode]{hyperref}
\usepackage[normalem]{ulem}

\usepackage{constants}
\usepackage{bbm}
\usepackage[utf8]{inputenc}
\usepackage{enumitem}
\usepackage{color}

\hypersetup{colorlinks=true, citecolor=blue, linkcolor=blue, urlcolor=blue, pdfstartview=FitH, pdfauthor=Jesse Thorner and Shifan Zhao, pdftitle=Landau--Siegel zeros of Rankin--Selberg $L$-functions}

\newtheorem{theorem}{Theorem}[section]
\newtheorem{corollary}[theorem]{Corollary}
\newtheorem{definition}[theorem]{Definition}
\newtheorem{proposition}[theorem]{Proposition}
\newtheorem{lemma}[theorem]{Lemma}
\theoremstyle{remark}
\newtheorem{remark}[theorem]{Remark}

\newcommand{\Sym}{\mathrm{Sym}}

\newcommand{\GL}{\mathrm{GL}}
\newcommand{\re}{\mathrm{Re}}

\newcommand{\R}{\mathbb{R}}
\newcommand{\CC}{\mathbb{C}}
\renewcommand{\bar}{\overline}
\renewcommand{\tilde}{\widetilde}
\newcommand{\fF}{\mathfrak{F}}
\newcommand{\fA}{\mathfrak{A}}

\numberwithin{equation}{section}

\newconstantfamily{abcon}{symbol=c}

\title[Landau--Siegel zeros of Rankin--Selberg $L$-functions]{Landau--Siegel zeros of Rankin--Selberg $L$-functions}
\author{Jesse Thorner}
\address{Department of Mathematics, University of Illinois, Urbana, IL 61801, USA}
\email{\href{mailto:jesse.thorner@gmail.com}{jesse.thorner@gmail.com}}
\author{Shifan Zhao}
\address{Department of Mathematics, The Ohio State University, Columbus, OH 43210, USA}
\email{\href{mailto:zhao.3326@buckeyemail.osu.edu}{zhao.3326@buckeyemail.osu.edu}}

\begin{document}

\begin{abstract}
We establish standard zero-free regions with no exceptional Landau--Siegel zeros for Rankin--Selberg $L$-functions and triple product $L$-functions in several new families for which modularity is not yet known.
\end{abstract}

\maketitle

\section{Introduction and statement of the main results}
\label{sec:intro}

Let $F$ be a number field, $\mathbb{A}_F$ the ring of adeles over $F$, and $\mathfrak{F}_{n}$ the set of cuspidal automorphic representations $\pi$ of $\mathrm{GL}_{n}(\mathbb{A}_F)$.  Let $\mathfrak{C}_{\pi}$ be the analytic conductor of $\pi$ (see \eqref{eqn:AC_def}), which captures the arithmetic and spectral complexity of $\pi$.  Let $L(s,\pi)$ be the standard $L$-function, $\widetilde{\pi}\in\mathfrak{F}_{n}$ the contragredient, and $\omega_{\pi}$ the central character of $\pi$ (which we normalize to be unitary).  The generalized Riemann hypothesis (GRH) asserts that if $\pi\in\mathfrak{F}_{n}$ and $\mathrm{Re}(s)>\frac{1}{2}$, then $L(s,\pi)\neq 0$.  Jacquet and Shalika~\cite{JS} proved that if $\mathrm{Re}(s)\geq 1$, then $L(s,\pi)\neq 0$, extending classical work on the Riemann zeta function.  Let $|\cdot|$ denote the idelic norm.  Replacing $\pi$ with $\pi\otimes|\cdot|^{it}$ and varying $t\in\R$, we find that it is equivalent to prove that if $\pi\in\mathfrak{F}_{n}$ and $\sigma\geq 1$, then $L(\sigma,\pi)\neq 0$.

In the absence of GRH, it is important for arithmetic applications that the zero-free region of $L(s,\pi)$ to have strong uniform dependence on $\mathfrak{C}_{\pi}$.  Classical techniques (e.g.,~\cite{Wattanawanichkul}) show that there is an absolute constant $\Cl[abcon]{ZFR1}>0$ such that $L(\sigma,\pi)$ has at most one real exceptional zero in the interval $\sigma\geq  1-\Cr{ZFR1}/(n\log \mathfrak{C}_{\pi})$.  This exceptional zero, which can only exist when $\pi=\tilde{\pi}$ (i.e., $\pi$ is {\it self-dual}), might be very close to $s=1$ as a function of $\mathfrak{C}_{\pi}$.  Many important problems depend on the elimination of this exceptional zero, which is sometimes called a Landau--Siegel zero.

Let $v$ be a place of $F$, and let $F_v$ be the completion of $F$ relative to $v$.  Given $\pi\in\mathfrak{F}_{n}$, we express $\pi$ as a restricted tensor product $\bigotimes_v\pi_v$ of smooth, admissible representations of $\mathrm{GL}_n(F_v)$.  There is a finite $S_{\pi}$ (possibly empty) set of places $v$ at which $\pi_v$ is ramified.  If $v\notin S_{\pi}$ is non-archimedean, then the Satake isomorphism attaches to $\pi_v$ a semisimple conjugacy class of $\GL_n(\CC)$ with representative $A(\pi_v)=\mathrm{diag}(\alpha_{1,\pi}(v),\ldots,\alpha_{n,\pi}(v))$.  The Langlands principle of functoriality predicts that if $r\colon \GL_2(\CC)\to \GL_n(\CC)$ is an algebraic representation, then there should be a map $\rho$ from automorphic representations of $\GL_2(\mathbb{A}_F)$ to automorphic representations of $\GL_n(\mathbb{A}_F)$, with compatible local maps, such that if $v\notin S_{\pi}$ is non-archimedean, then $r(A(\pi_v))=A(\rho(\pi)_v)$.  In order to establish the principle of functoriality for all representations $r$, it suffices to establish it for irreducible $r$.

Let $\pi \in \fF_2$.  For $m\geq 0$, let $\mathrm{Sym}^m\colon \mathrm{GL}_2(\mathbb{C})\to\mathrm{GL}_{m+1}(\mathbb{C})$ be the $(m+1)$-dimensional irreducible representation of $\mathrm{GL}_2(\mathbb{C})$ on symmetric tensors of rank $m$.  If $P(x,y)$ is a homogeneous degree $m$ polynomial in two variables and $g\in\mathrm{GL}_2(\mathbb{C})$, then $\mathrm{Sym}^m(g)\in\mathrm{GL}_{m+1}(\mathbb{C})$ is the matrix giving the change in coefficients of $P$ under the change of variables by $g$.  Let $\varphi_v$ be the two-dimensional representation of the Deligne--Weil group attached to $\pi_v$ and $\mathrm{Sym}^m(\pi_v)$ be the smooth admissible representation of $\mathrm{GL}_{m+1}(F_v)$ attached to the representation $\mathrm{Sym}^m\circ \varphi_v$.  By the local Langlands correspondence, $\mathrm{Sym}^m(\pi_v)$ is well-defined for every place $v$, so we can define the Euler product associated to the $m$-th symmetric power lift of $\pi$:  If $\chi\in\mathfrak{F}_1$, then
\[
L(s,\pi,\mathrm{Sym}^m\otimes\chi) = \prod_{v\nmid\infty} L(s,\mathrm{Sym}^m(\pi_v)\otimes\chi_v).
\]
If $r_0$ is the standard representation of $\GL_2(\CC)$ with determinant $L$, then for each irreducible representation $r$ of $\GL_2(\CC)$, there exist integers $n\geq 0$ and $k$ such that $r=\Sym^n(r_0)\otimes L^{\otimes k}$.  The principle of functoriality now predicts that $\mathrm{Sym}^m(\pi) = \bigotimes_v \mathrm{Sym}^m(\pi_v)$ is an automorphic representation of $\mathrm{GL}_{m+1}(\mathbb{A}_F)$.  This is known for $m\leq 4$~\cite{GJ,Kim,KimShahidi}.  If $\mathrm{Sym}^m(\pi)$ is known to be automorphic, then we write $L(s,\mathrm{Sym}^m(\pi)\otimes\chi)$ instead of $L(s,\pi,\mathrm{Sym}^m\otimes\chi)$.

Let $L(s,\pi\times\pi')$ be the Rankin--Selberg $L$-function associated to $(\pi,\pi')\in\mathfrak{F}_{n}\times\mathfrak{F}_{n'}$ (see Jacquet, Piatetski-Shapiro, and Shalika~\cite{JPSS}).  Shahidi~\cite{Shahidi} proved that $L(s,\pi\times\pi')\neq 0$ when $\mathrm{Re}(s)\geq 1$.  Equivalently, replacing $\pi$ with $\pi\otimes|\cdot|^{it}$ and varying $t\in\mathbb{R}$, we have that if $(\pi,\pi')\in\mathfrak{F}_{n}\times\mathfrak{F}_{n'}$ and $\sigma\geq 1$, then $L(\sigma,\pi\times\pi')\neq 0$.  Brumley~\cite{Brumley,Lapid} proved that there exists an effectively computable constant $\Cl[abcon]{Brumley}=\Cr{Brumley}(n,n',F)>0$ such that $L(\sigma,\pi\times\pi')\neq 0$ when $\sigma\geq 1-1/(\mathfrak{C}_{\pi}\mathfrak{C}_{\pi'})^{\Cr{Brumley}}$.  See also the related recent work of Harcos and Thorner \cite{HarcosThorner,HarcosThorner2}.

The principle of functoriality also asserts that $L(s,\pi\times\pi')$ factors as a product of standard $L$-functions (i.e., it is {\it modular}).  Hoffstein and Ramakrishnan~\cite{HoffsteinRamakrishnan} proved that if all Rankin--Selberg $L$-functions are modular and $\pi \in \cup_{n=2}^{\infty}\mathfrak{F}_{n}$, then $L(s,\pi)$ has no exceptional zero (see Section \ref{subsec:earlier_work} for a more detailed discussion).  Modularity for $L(s,\pi\times\pi')$ is known only in special cases, most notably when $\pi\in\mathfrak{F}_{2}$ and $\pi'\in\mathfrak{F}_{2}\cup\mathfrak{F}_{3}$~\cite{KimShahidi,Ramakrishnan}. Therefore, the unconditional elimination of exceptional zeros remains a difficult and fruitful problem.  We say that $L(s,\pi\times\pi')$ has {\it no exceptional zero}, or {\it no Landau--Siegel zero}, if there exists an absolute and effectively computable constant $\Cl[abcon]{ZFR11}>0$ such that
\[
L(\sigma,\pi\times\pi')\neq 0,\qquad \sigma \geq 1-\Cr{ZFR11}/(n n'\log(\mathfrak{C}_{\pi}\mathfrak{C}_{\pi'})).
\]
If $(\pi,\pi')\in\mathfrak{F}_2$ and $r,r'$ are algebraic representations such that $r(\pi)$ and $r'(\pi')$ are automorphic, then we say that $L(s,r(\pi)\times r'(\pi'))$ has {\it no exceptional zero}, or {\it no Landau--Siegel zero}, if there exists an effectively computable constant $\Cl[abcon]{ZFR_r}=\Cr{ZFR_r}(r,r')>0$ such that
\[
L(\sigma,r(\pi)\times r'(\pi'))\neq 0,\qquad \sigma \geq 1-\Cr{ZFR_r}/\log(\mathfrak{C}_{\pi}\mathfrak{C}_{\pi'}).
\]

We call $\pi,\pi'\in\mathfrak{F}_{n}$ {\it twist-equivalent}, denoted $\pi\sim\pi'$, if there exists $\psi\in\mathfrak{F}_{1}$ such that $\pi'=\pi\otimes\psi$.  Otherwise, we write $\pi\not\sim\pi'$.  Let $\mathbbm{1}\in\mathfrak{F}_{1}$ be the trivial representation, whose $L$-function is the Dedekind zeta function $\zeta_F(s)$.  If $\chi\in\mathfrak{F}_1$, $\chi=\bar{\chi}$, and $L(s,\chi)$ appears in as a factor of another $L$-function, then we call $L(s,\chi)$ a {\it self-dual abelian factor}.  If $\pi\in\mathfrak{F}_2$ and there exists a nontrivial quadratic $\eta\in\mathfrak{F}_1$ such that $\pi = \pi \otimes \eta$, then $\pi$ is \textit{dihedral}.

\begin{proposition}
\label{prop:list}
\begin{enumerate}[leftmargin=*]
	\item~\cite{HoffsteinRamakrishnan,Wattanawanichkul} If $\pi\in\cup_{n=1}^{\infty}\mathfrak{F}_{n}$ and $\pi\neq \tilde{\pi}$, then $L(s,\pi)$ has no exceptional zero.
	\item~\cite{Banks,HoffsteinRamakrishnan} If $\pi\in\mathfrak{F}_{2}\cup \mathfrak{F}_{3}$, then $L(s,\pi)$ has no exceptional zero.
	\item~\cite{Luo} If $\pi\in\mathfrak{F}_n$ and $\pi\otimes\psi=\pi$ for some $\psi\in\mathfrak{F}_1-\{\mathbbm{1}\}$, then $L(s,\pi)$ has no exceptional zero.
	\item~\cite{RamakrishnanWang} If $\pi,\pi'\in\mathfrak{F}_{2}$, then any exceptional zero of $L(s,\pi\times\pi')$ is a zero of a self-dual abelian factor.  No such factor exists when $\pi\not\sim\pi'$ and at least one of $\pi,\pi'$ is non-dihedral.
	\item~\cite{RamakrishnanWang} If $\pi\in\mathfrak{F}_{2}$ is self-dual, then any exceptional zero of $L(s,\mathrm{Sym}^2(\pi)\times\mathrm{Sym}^2(\pi))/\zeta_F(s)=L(s,\mathrm{Sym}^2(\pi)\otimes\omega_{\pi})L(s,\mathrm{Sym}^4(\pi))$ is a zero of a self-dual abelian factor.
	\item~\cite{Luo} Let $(\pi,\pi')\in\mathfrak{F}_2\times\mathfrak{F}_3$.  If $\pi'\not\sim\Sym^2(\pi)$ or $\pi$ is dihedral, then $L(s,\pi\times\pi')$ has no exceptional zero.
	\item~\cite{HumphriesThorner,Wattanawanichkul} Let $(\pi,\pi')\in\mathfrak{F}_n\times\mathfrak{F}_{n'}$.  Suppose that $L(s,\pi\times\pi')=L(s,\tilde{\pi}\times\tilde{\pi}')$.  If $t\neq 0$, then $L(s,\pi\times(\pi'\otimes|\cdot|^{it}))$ has no exceptional zero.
	\item~\cite{Humphries,Wattanawanichkul} If $(\pi,\pi')\in\mathfrak{F}_n\times\mathfrak{F}_{n'}$, $\pi\neq\tilde{\pi}$, and $\pi'=\tilde{\pi}'$, then $L(s,\pi\times\pi')$ has no exceptional zero.
\end{enumerate}
\end{proposition}

When $F$ is totally real and $\pi,\pi'\in\mathfrak{F}_2$ are non-dihedral and regular algebraic (so that they correspond with holomorphic primitive Hilbert cusp forms), Newton and Thorne~\cite{NewtonThorne3,NewtonThorne,NewtonThorne2} proved that if $n\geq 1$, then $\Sym^n(\pi),\Sym^n(\pi')\in\mathfrak{F}_{n+1}$.  Using this, Thorner proved the following result.

\begin{proposition}[{\cite[Theorem 1.1]{Thorner_Siegel}}]
\label{prop:Thorner}
Let $F$ be totally real.  Let $\pi,\pi'\in\mathfrak{F}_2$ be non-dihedral and regular algebraic.  If $m,n\geq 0$, $m+n\geq 1$, and $\chi\in\mathfrak{F}_1$ corresponds with a ray class character over $F$, then any exceptional zero of $L(s,\Sym^m(\pi)\times(\Sym^n(\pi')\otimes\chi))$ is a zero of a self-dual abelian factor.  No such factor exists when $m\neq n$ or $\pi\not\sim\pi'$.
\end{proposition}

In this paper, we eliminate exceptional zeros for new families of Rankin--Selberg $L$-functions and triple product $L$-functions.

\begin{theorem}
\label{thm:main}
Let $\chi\in\mathfrak{F}_1$.  Let $\pi,\pi',\pi''\in\mathfrak{F}_2$ be non-dihedral and pairwise twist-inequivalent.  Let $\pi_0\in\mathfrak{F}_3$ satisfy $\pi_0\not\sim\Sym^2(\pi)$ and $\pi_0\not\sim\Sym^2(\pi')$.
\begin{enumerate}[leftmargin=*]
\item The following $L$-functions have no exceptional zeros:
\begin{enumerate}[leftmargin=*]
	\item $L(s,\mathrm{Sym}^2(\pi)\times\pi_0)$,
	\item $L(s,\mathrm{Sym}^3(\pi)\times(\mathrm{Sym}^2(\pi')\otimes\chi))$,
	\item $L(s,\mathrm{Sym}^4(\pi)\times\pi')$,
	\item $L(s,\pi\times\pi'\times\pi'')$,
	\item $L(s,\pi\times\pi'\times\pi_0)$, and
	\item $L(s,\pi\times\mathrm{Sym}^2(\pi')\times\mathrm{Sym}^2(\pi''))$.
\end{enumerate}
\item  If $\mathrm{Sym}^3(\pi)\neq\mathrm{Sym}^3(\pi')\otimes \overline{\omega}_{\pi'}^2$, then $L(s,\mathrm{Sym}^3(\pi)\times \pi')$ has no exceptional zero.
\item If $\mathrm{Sym}^4(\pi)\otimes\bar{\omega}_{\pi}^2\neq\mathrm{Sym}^4(\pi')\otimes\bar{\omega}_{\pi'}^2$ or $\chi^2\omega_{\pi}^4\omega_{\pi'}^2\neq\mathbbm{1}$, then $L(s,\mathrm{Sym}^4(\pi)\times(\mathrm{Sym}^2(\pi')\otimes\chi))$ has no exceptional zero.
\end{enumerate}
\end{theorem}
\begin{remark}
\begin{enumerate}
    \item Unlike Proposition~\ref{prop:Thorner}, the work of Newton and Thorne~\cite{NewtonThorne3,NewtonThorne,NewtonThorne2} does not apply in our level of generality.  In particular, we do not require that $F$ be totally real, and we permit $\pi,\pi'\in\mathfrak{F}_2$ to correspond with Hecke--Maa{\ss} newforms.
    \item Unlike Proposition~\ref{prop:list}(4,5,6), the $L$-functions considered are not yet known to be modular.
    \item Unlike Proposition~\ref{prop:list}(7,8), there are no hypotheses regarding self-duality.
    \item For readers already familiar with the proofs of Propositions~\ref{prop:list}~and~\ref{prop:Thorner}, we summarize our strategy and compare it with earlier results in Sections~\ref{subsec:earlier_work}~and~\ref{subsec:OurApproach}.
    \item The Euler products for the $L$-functions we consider have degree between $8$ and $18$.
    \item In Parts (2) and (3), the hypothesis is satisfied, for example, when
	\begin{enumerate}
	\item[(i)] there is a non-archimedean place $v$ of $F$ at which exactly one of $\pi_v$ and $\pi_v'$ is ramified, or
	\item[(ii)] $\pi$ and $\pi'$ are regular algebraic (regardless of whether $F$ is totally real).
	\end{enumerate}
\end{enumerate}
\end{remark}

Our proof of Theorem~\ref{thm:main} accounts for more possibilities than we have stated (e.g., at least one of $\pi,\pi',\pi''$ is dihedral, at least two are twist-equivalent, etc.).  Within our exhaustive casework, there are several cases where any exceptional zero must be a zero of a self-dual abelian factor.  We show how to classify the possible self-dual abelian factors when they exist, though we do not always make this classification fully explicit.  To help with this classification, we need two additional results on the nonexistence of exceptional zeros.  The first generalizes Proposition~\ref{prop:list}(5) (see~\cite{RamakrishnanWang}), and the second generalizes Proposition~\ref{prop:list}(3) (see~\cite{Luo}).

\begin{theorem}
\label{symmetric fourth no Siegel zero}
If $\pi\in \mathfrak{F}_2$ and $\chi\in \mathfrak{F}_1$, then any exceptional zero of $L(s,\Sym^4(\pi)\otimes\chi)$ is a zero of a self-dual abelian factor.  No such factor exists when $\pi$ is octahedral or not of solvable polyhedral type.
\end{theorem}
\begin{remark}
See Section \ref{subsec:symm} for the definitions of ``octahedral'' and ``solvable polyhedral type.''	
\end{remark}

\begin{theorem}
\label{thm:main3}
Let $\pi\in\mathfrak{F}_n$ and $\pi'\in\mathfrak{F}_{n'}$.  If there exists a nontrivial $\psi\in\mathfrak{F}_1$ such that $\pi=\pi\otimes\psi$ and $\pi'\neq\pi'\otimes\psi$, then $L(s,\pi\times\pi')$ has no exceptional zero.	
\end{theorem}

\subsection*{An application}

In each of our theorems, we can replace $\pi$ with $\pi\otimes|\cdot|^{it}$ and let $t\in\R$ vary.  So doing, our theorems produce {\it standard zero-free regions} with no exceptional zero. This makes the results described above useful in the theory of primes.  Let $F$ be a totally real number field, and let $\pi,\pi',\pi''\in\mathfrak{F}_2$ be non-dihedral, pairwise twist-inequivalent, and regular algebraic.   The motivation of Thorner in~\cite{Thorner_Siegel} for proving Proposition~\ref{prop:Thorner} was to obtain a highly uniform rate of convergence in the Sato--Tate law for the Hecke eigenvalues of $\pi$ and the joint Sato--Tate law for the Hecke eigenvalues of $\pi$ and $\pi'$.  One possible application of Theorem~\ref{thm:main} is the existence of an absolute constant $\Cl[abcon]{ST3}>0$, and nontrivial proper subintervals $I,I',I''\subseteq[-2,2]$, and a non-archimedean place $v$ at which $\pi_v$, $\pi_v'$, and $\pi_v''$ are unramified such that $q_v\leq (\mathfrak{C}_{\pi}\mathfrak{C}_{\pi'}\mathfrak{C}_{\pi''})^{\Cr{ST3}}$ and the Hecke eigenvalues at $v$ satisfy $(a_{\pi}(v),a_{\pi'}(v),a_{\pi''}(v))\in I\times I'\times I''$.  The authors plan to explore this in another paper.

\subsection*{Organization}
In Sections~\ref{sec:Properties}~and~\ref{sec:symm}, we cite the properties of automorphic representations and $L$-functions that we need for our proofs.

In Section~\ref{sec:proofs_Sym4+self-twist}, we prove Theorems~\ref{symmetric fourth no Siegel zero}~and~\ref{thm:main3}.

In Section~\ref{sec:strategy}, we summarize our strategy for proving Theorem~\ref{thm:main} and contrast it with earlier approaches.

In Section~\ref{sec:1a}, we prove Theorem~\ref{thm:GL3xGL3}, from which we deduce Theorem~\ref{thm:main}(1a).

In Section~\ref{sec:1b}, we prove Theorem~\ref{thm:GL4xGL3}, from which we deduce Theorem~\ref{thm:main}(1b).

In Section~\ref{sec:1c}, we prove Theorem~\ref{thm:GL5xGL2}, from which we deduce Theorem~\ref{thm:main}(1c).

In Section~\ref{sec:1d}, we prove Theorem~\ref{thm:GL2xGL2xGL2}, from which we deduce Theorem~\ref{thm:main}(1d).

In Section~\ref{sec:1e}, we prove Theorem~\ref{thm:GL2xGL2xGL3}, from which we deduce Theorem~\ref{thm:main}(1e).

In Section~\ref{sec:1f}, we prove Theorem~\ref{thm:GL2xGL3xGL3}, from which we deduce Theorem~\ref{thm:main}(1f).

In Section~\ref{sec:2}, we prove Theorem~\ref{thm:GL2xGL4}, from which we deduce Theorem~\ref{thm:main}(2).

In Section~\ref{sec:3}, we prove Theorem~\ref{thm:GL5xGL3}, from which we deduce Theorem~\ref{thm:main}(3).

\subsection*{Acknowledgements}

The authors thank Jeffrey Hoffstein, Wenzhi Luo, and Djordje Mili{\'c}evi{\'c} for helpful conversations.  JT is partially funded by the Simons Foundation (MP-TSM-00002484) and the National Science Foundation (DMS-2401311).

\section{Analytic properties of $L$-functions}
\label{sec:Properties}

Let $F$ be a number field with absolute discriminant $D_F$.  For a place $v$ of $F$, let $F_v$ be the associated completion.  For each place $v$ of $F$, we write $v\mid\infty$ (resp. $v\nmid \infty$) if $v$ is archimedean (resp. non-archimedean).  For $v\nmid\infty$, $q_v$ is the cardinality of the residue field of the local ring of integers $\mathcal{O}_v\subseteq F_v$, and $\varpi_v$ is the uniformizer.  The properties of $L$-functions given here rely on~\cite{GodementJacquet,JPSS,MoeglinWaldspurger}.

\subsection{Standard $L$-functions}
\label{subsec:standard}

Let $\pi\in\mathfrak{F}_{n}$, let $\widetilde{\pi}\in\mathfrak{F}_{n}$ be the contragredient, and let $\omega_{\pi}$ the central character.  We express $\pi$ as a restricted tensor product $\bigotimes_v \pi_v$ of smooth admissible representations of $\mathrm{GL}_n(F_v)$.  Let $\delta_{\pi}=1$ if $\pi=\mathbbm{1}$ and $\delta_{\pi}=0$ otherwise.  Define the sets $S_{\pi}=\{v\nmid\infty\colon \textup{$\pi_v$ ramified}\}$ and $S_{\pi}^{\infty}=S_{\pi}\cup\{v\mid\infty\}$.  Let $N_{\pi}$ be the norm of the conductor of $\pi$.  If $v\nmid\infty$, then there are $n$ Satake parameters $(\alpha_{j,\pi}(v))_{j=1}^n$ such that
	\[
	L(s,\pi)=\prod_{v\nmid\infty}L(s,\pi_{v}),\qquad L(s,\pi_{v}) = \prod_{j=1}^{n}\frac{1}{1-\alpha_{j,\pi}(v)q_v^{-s}}
	\]
	converges absolutely for $\mathrm{Re}(s)>1$.  If $v\in S_{\pi}$, then at least one of the $\alpha_{j,\pi}(v)$ equals zero.
	
If $v\mid\infty$, then $(\mu_{j,\pi}(v))_{j=1}^n$ are the Langlands parameters at $v$, from which we define
\[
\Gamma_v(s)=\begin{cases}
	\pi^{-s/2}\Gamma(s/2)&\mbox{if $F_v=\mathbb{R}$,}\\
	2(2\pi)^{-s}\Gamma(s)&\mbox{if $F_v=\mathbb{C}$,}
	\end{cases}\qquad L(s,\pi_{\infty}) = \prod_{v\mid\infty}L(s,\pi_v) = \prod_{v\mid\infty}\prod_{j=1}^n \Gamma_v(s+\mu_{j,\pi}(v)).
	\]
	The completed $L$-function
	\[
	\Lambda(s,\pi)=(s(1-s))^{\delta_{\pi}}(D_F^n N_{\pi})^{\frac{s}{2}}L(s,\pi)L(s,\pi_{\infty})
	\]
	is entire of order $1$, and there exists a complex number $W(\pi)$ of modulus $1$ such that
	\[
	\Lambda(s,\pi)=W(\pi)\Lambda(1-s,\tilde{\pi}).
	\]
	Since $\{\alpha_{j,\widetilde{\pi}}(v)\}=\{\overline{\alpha_{j,\pi}(v)}\}$, $N_{\widetilde{\pi}}=N_{\pi}$, and $\{\mu_{j,\widetilde{\pi}}(v)\}=\{\overline{\mu_{j,\pi}(v)}\}$, we have that $L(s,\widetilde{\pi})=\overline{L(\overline{s},\pi)}$.  The analytic conductor is
\begin{equation}
\label{eqn:AC_def}
\mathfrak{C}_{\pi}=D_F^n N_{\pi} \prod_{v\mid\infty}\prod_{j=1}^n (|\mu_{j,\pi}(v)|+3)^{[F_v:\mathbb{R}]}.
\end{equation}
By~\cite{LRS,MullerSpeh} there exists $\theta_n\in[0,\frac{1}{2}-\frac{1}{n^2+1}]$ such that
\begin{equation}
\label{eqn:Ramanujan1}
|\alpha_{j,\pi}(v)|\leq q_v^{\theta_{n}},\qquad \mathrm{Re}(\mu_{j,\pi}(v))\geq -\theta_{n}.
\end{equation}
We define $a_{\pi}(v^{\ell})$ by the Dirichlet series identity
\[
-\frac{L'}{L}(s,\pi)=\sum_{v \nmid\infty}\sum_{\ell=1}^{\infty}\frac{\sum_{j=1}^n \alpha_{j,\pi}(v)^{\ell}\log q_v}{q_v^{\ell s}}=\sum_{v\nmid\infty}\sum_{\ell=1}^{\infty}\frac{a_{\pi}(v^{\ell})\log q_v}{q_v^{\ell s}}.
\]

\subsection{Rankin--Selberg $L$-functions}
\label{subsec:RS}

Let $\pi\in\mathfrak{F}_{n}$ and $\pi'\in\mathfrak{F}_{n'}$.  Let
\[
\delta_{\pi\times\pi'}=\begin{cases}
1&\mbox{if $\pi'=\widetilde{\pi}$,}\\
0&\mbox{otherwise.}
\end{cases}
\]
For each $v\notin S_{\pi}^{\infty}\cup S_{\pi'}^{\infty}$, define
\[
L(s,\pi_{v}\times\pi_{v}')=\prod_{j=1}^n \prod_{j'=1}^{n'}\frac{1}{1-\alpha_{j,\pi}(v)\alpha_{j',\pi'}(v)q_v^{-s}}.
\]
Jacquet, Piatetski-Shapiro, and Shalika proved the following theorem.

\begin{theorem}\cite{JPSS}
\label{thm:JPSS}
If $(\pi,\pi')\in\mathfrak{F}_{n}\times\mathfrak{F}_{n'}$, then there exist
\begin{enumerate}
\item complex numbers $(\alpha_{j,j',\pi\times\pi'}(v))_{j=1}^n{}_{j'=1}^{n'}$ for each $v\in S_{\pi}\cup S_{\pi'}$, from which we define
	\begin{align*}
	L(s,\pi_v\times\pi_v') = \prod_{j=1}^n \prod_{j'=1}^{n'}\frac{1}{1-\alpha_{j,j',\pi\times\pi'}(v)q_v^{-s}},\quad L(s,\widetilde{\pi}_v\times\widetilde{\pi}_v') = \prod_{j=1}^n \prod_{j'=1}^{n'}\frac{1}{1-\overline{\alpha_{j,j',\pi\times\pi'}(v)}q_v^{-s}};
	\end{align*}
\item complex numbers $(\mu_{j,j',\pi\times\pi'}(v))_{j=1}^{n}{}_{j'=1}^{n'}$ for each $v\mid\infty$, from which we define
	\begin{align*}
	L(s,\pi_{v}\times\pi_{v}') = \prod_{j=1}^n \prod_{j'=1}^{n'}\Gamma_v(s+\mu_{j,j',\pi\times\pi'}(v)),\quad L(s,\widetilde{\pi}_{v}\times\widetilde{\pi}_{v}') = \prod_{j=1}^n \prod_{j'=1}^{n'}\Gamma_v(s+\overline{\mu_{j,j',\pi\times\pi'}(v)});
	\end{align*}
\item a conductor, an integral ideal whose norm is denoted $N_{\pi\times\pi'}=N_{\widetilde{\pi}\times\widetilde{\pi}'}$; and
\item a complex number $W(\pi\times\pi')$ of modulus $1$
\end{enumerate}
such that the Rankin--Selberg $L$-functions
\[
L(s,\pi\times\pi')=\prod_{v\nmid\infty}L(s,\pi_v\times\pi_v'),\qquad L(s,\widetilde{\pi}\times\widetilde{\pi}')=\prod_{v\nmid\infty}L(s,\widetilde{\pi}_v\times\widetilde{\pi}_v')
\]
converge absolutely for $\mathrm{Re}(s)>1$, the completed $L$-functions
\begin{align*}
\Lambda(s,\pi\times\pi') &= (s(1-s))^{\delta_{\pi\times\pi'}} (D_F^{n'n}N_{\pi\times\pi'})^{\frac{s}{2}} L(s,\pi\times\pi')\prod_{v\mid\infty}L(s,\pi_v\times\pi_{v}')\\
\Lambda(s,\widetilde{\pi}\times\widetilde{\pi}') &= (s(1-s))^{\delta_{\pi\times\pi'}} (D_F^{n'n}N_{\pi\times\pi'})^{\frac{s}{2}}L(s,\widetilde{\pi}\times\widetilde{\pi}')\prod_{v\mid\infty}L(s,\widetilde{\pi}_{v}\times\widetilde{\pi}_{v}')
\end{align*}
are entire of order $1$, and $\Lambda(s,\pi\times\pi')=W(\pi\times\pi')\Lambda(1-s,\widetilde{\pi}\times\widetilde{\pi}')$.
\end{theorem}

It follows from Theorem~\ref{thm:JPSS} that
\begin{equation}
	\label{eqn:dual}
	L(s,\widetilde{\pi}\times\widetilde{\pi}')=\overline{L(\overline{s},\pi\times\pi')}
	\end{equation}
The following bounds hold:
	\begin{equation}
		\label{eqn:Ramanujan2}
		\begin{aligned}
	|\alpha_{j,j',\pi\times\pi'}(v)|\leq  q_v^{\theta_n+\theta_{n'}},\qquad \mathrm{Re}(\mu_{j,j',\pi\times\pi'}(v))\geq -(\theta_n+\theta_{n'}).
	\end{aligned}
	\end{equation}
	If $\ell\geq 1$ is an integer and $v\nmid\infty$, then we define
	\begin{equation}
	\label{eqn:a_def}
	\begin{aligned}
	a_{\pi\times\pi'}(v^{\ell})&= \begin{cases}
 	a_{\pi}(v^{\ell})a_{\pi'}(v^{\ell})&\mbox{if $v\notin S_{\pi}\cup S_{\pi'}$,}\\
 	\sum_{j=1}^n \sum_{j'=1}^{n'}\alpha_{j,j',\pi\times\pi'}(v)^{\ell}&\mbox{if $v\in S_{\pi}\cup S_{\pi'}$,}
 \end{cases}\\
 a_{\widetilde{\pi}\times\widetilde{\pi}'}(v^{\ell})&=\overline{a_{\pi\times\pi'}(v^{\ell})}.
 \end{aligned}
	\end{equation}
	We have the Dirichlet series identity
	\begin{equation}
		\label{eqn:log_deriv}
	-\frac{L'}{L}(s,\pi\times\pi')=\sum_{v\nmid\infty}\sum_{\ell=1}^{\infty}\frac{a_{\pi\times\pi'}(v^{\ell})\log q_v}{q_v^{\ell s}},\qquad\mathrm{Re}(s)>1.
	\end{equation}
	
\subsection{Isobaric sums}
\label{subsec:isobaric}

Let $r\geq 1$ be an integer.  For $1\leq j\leq r$, let $\pi_{j}\in\mathfrak{F}_{d_j}$.  Langlands associated to $(\pi_1,\ldots,\pi_r)$ an automorphic representation of $\mathrm{GL}_{d_1+\cdots+d_r}(\mathbb{A}_F)$, the isobaric sum $\Pi=\pi_1\boxplus\cdots\boxplus\pi_r$.  Its $L$-function is
\[
L(s,\Pi)=\prod_{j=1}^r L(s,\pi_j),
\]
and its contragredient is $\widetilde{\pi}_1\boxplus\cdots\boxplus\widetilde{\pi}_r$.  Let $\mathfrak{A}_{n}$ be the set of isobaric automorphic representations of $\mathrm{GL}_n(\mathbb{A}_F)$.  If $\Pi=\pi_1\boxplus\cdots\boxplus\pi_r\in\mathfrak{A}_n$ and $\Pi'=\pi_1'\boxplus\cdots\boxplus\pi_{r'}'\in\mathfrak{A}_{n'}$, then
\begin{equation}
\label{eqn:isobaric_RS}
L(s,\Pi\times\Pi')=\prod_{j=1}^r \prod_{k=1}^{r'} L(s,\pi_j\times\pi_k').
\end{equation}
\begin{lemma}\cite{HoffsteinRamakrishnan}
\label{lem:HR}
If $\Pi\in\mathfrak{A}_{n}$, then $-\frac{L'}{L}(s,\Pi\times\widetilde{\Pi})$ has non-negative Dirichlet coefficients.
\end{lemma}

\subsection{Real zeros}
	
We define the analytic conductor
	\[
	\mathfrak{C}_{\pi\times\pi'}=D_F^{n'n}N_{\pi\times{\pi}'}\prod_{v\mid\infty}\prod_{j=1}^n \prod_{j'=1}^{n'} (|\mu_{j,j',\pi\times{\pi}'}(v)|+3)^{[F_v:\mathbb{R}]}.
	\]
	Using~\cite[Theorem 2]{BushnellHenniart} and~\cite[Lemma~A.1]{Wattanawanichkul}, we infer that
	\begin{equation}
		\label{eqn:BH}
		N_{\pi\times\pi'}\mid N_{\pi}^{n'}N_{\pi'}^{n},\qquad \mathfrak{C}_{\pi\times\pi'}\leq \mathfrak{C}_{\pi}^{n'}\mathfrak{C}_{\pi'}^{n},\qquad \mathfrak{C}_{\pi\times(\pi'\otimes\chi)}\leq \mathfrak{C}_{\pi}^{n'}\mathfrak{C}_{\pi'}^{n}\mathfrak{C}_{\chi}^{n'n}.
	\end{equation}

\begin{lemma}
	\label{lem:GHL}
Let $J\geq 1$.  For $j\in\{1,\ldots,J\}$, let $(\pi_j,\pi_j',\chi_j)\in\mathfrak{F}_{n_j}\times\mathfrak{F}_{n_j'}\times\mathfrak{F}_{1}$.  Define
\[
\mathfrak{Q} = \prod_{j = 1}^
   J C (\pi_j) C (\pi_j') C (\chi),\quad \mathfrak {S} = \bigcup_ {j = 1}^J (S_{\pi_j}\cup S_{\pi_j'}\cup S_{\chi_j}),\quad D(s) =  \prod_{j = 1}^J L(s, \pi_j\times(\pi' _j\otimes\chi_j)).
\]
Assume that $D(s)$ is holomorphic on $\mathbb{C}-\{1\}$ with a pole of order $r\geq 1$ at $s=1$.  Write
\begin{equation}
\label{eqn:aDdef}
a_D(v^{\ell})=\sum_{j=1}^J a_{\pi_j\times(\pi_j'\otimes\chi)}(v^{\ell}),\qquad -\frac{D'}{D}(s) = \sum_{v\nmid\infty}\sum_{\ell=1}^{\infty}\frac{a_D(v^{\ell})\log q_v}{q_v^{\ell s}}.
\end{equation}
Let $Q \geq \mathfrak{Q}$.  There exists an effectively computable constant $\Cl[abcon]{GHL}>0$ (depending only on the numbers $n_j$, $n_j'$, and $r$) such that if $\mathrm{Re}(a_D(v^{\ell}))\geq 0$ for all $\ell\geq 1$ and $v\notin\mathfrak{S}$, then $D(\sigma)$ has no zeros in the interval $[1,\infty)$ and at most $r$ zeros in the interval $[1-\Cr{GHL}/\log Q,1)$.
\end{lemma}
\begin{proof}
The proof is identical to~\cite[Lemma~5.9]{IK} except that we estimate the contribution from the $a_D(v^{\ell})$ due to $v\in\mathfrak{S}$ using \eqref{eqn:Ramanujan1} and \eqref{eqn:Ramanujan2}.
\end{proof}

\section{Symmetric power lifts}
\label{sec:symm}

\subsection{Symmetric power lifts from $\mathrm{GL}_2$}
\label{subsec:symm}

Let $\pi\in\mathfrak{F}_{2}$.  Define $\mathrm{Sym}^0(\pi)=\mathbbm{1}$ and $\mathrm{Sym}^1(\pi)=\pi$.  For an integer $m\geq 2$ and a place $v$ of $F$, recall the definition of $\Sym^m(\pi_v)$ from Section \ref{sec:intro}, and let
\[
\Sym^m(\pi)=\bigotimes_v \Sym^m(\pi_v).
\]
It is conjectured that $\mathrm{Sym}^{m}(\pi)\in\mathfrak{A}_{m+1}$ with contragredient
\begin{equation}
\label{eqn:contra_sym}
\mathrm{Sym}^{m}(\widetilde{\pi})=\mathrm{Sym}^{m}(\pi)\otimes\overline{\omega}_{\pi}^{m},
\end{equation}
and this conjecture is known for $m\leq 4$.  For $\pi\in\mathfrak{F}_2$, we introduce
\[
A^0(\pi)=\mathbbm{1},~~~A^1(\pi)=\pi,~~~A^2(\pi)=\Sym^2(\pi)\otimes\overline{\omega}_{\pi},~~~A^3(\pi)=\Sym^3(\pi)\otimes\overline{\omega}_{\pi},~~~A^4(\pi)=\Sym^4(\pi)\otimes\overline{\omega}_{\pi}^2.
\]
Using \eqref{eqn:contra_sym}, we record the identities
\[
\tilde{A^j(\pi)}=A^j(\tilde{\pi}),\quad A^1(\tilde{\pi})=\pi\otimes\overline{\omega}_{\pi},\quad A^2(\tilde{\pi})=A^2(\pi),\quad A^3(\tilde{\pi})=A^3(\pi)\otimes\overline{\omega}_{\pi},\quad A^4(\tilde{\pi})=A^4(\pi).
\]

\begin{definition}
Let $\pi\in\mathfrak{F}_2$.
\begin{enumerate}[leftmargin=*]
    \item If there exists a non-trivial quadratic character $\eta=\eta_{\pi}\in\mathfrak{F}_1$ such that $\pi = \pi \otimes \eta$, then $\pi$ is {\bf dihedral}.  If $K/F$ is the quadratic extension associated with $\eta$, then there exists $\xi=\xi_{\pi}$ defined over $K$ such that $\pi=I_K^F(\xi)$, the automorphic induction of $\xi$ from $K$ to $F$.  Let $\theta \in \mathrm{Gal}(K/F)$ be the non-trivial element, and set $\xi' := \xi \circ \theta$. In this case, $\pi$ is {\bf dihedral (by $(\eta,\xi,K)$)}.
    \item If $\pi$ is not dihedral and there exists a non-trivial cubic character $\mu=\mu_{\pi}\in\mathfrak{F}_1$ such that $\Sym^2(\pi) = \Sym^2(\pi) \otimes \mu$, then $\pi$ is {\bf tetrahedral (by $\mu$)}.
    \item If $\pi$ is not dihedral or tetrahedral, and there exists a non-trivial quadratic character $\eta=\eta_{\pi}$ such that $\Sym^3(\pi) = \Sym^3(\pi) \otimes \eta$, then $\pi$ is {\bf octahedral (by $\eta$)}.
    \item If $\pi$ is not dihedral, tetrahedral, or octahedral, then $\pi$ is {\bf not of solvable polyhedral type}.
\end{enumerate}
\end{definition}

Throughout our proofs, we will use the following classification result.

\begin{lemma}
\label{symmetric power decomposition}\cite{GJ,Kim,KimShahidi2,KimShahidi,Ramakrishnan}
\begin{enumerate}[leftmargin=*]
	\item If $\pi\in\mathfrak{F}_{2}$, then $\mathrm{Sym}^2(\pi)\in\mathfrak{A}_{3}$, $\mathrm{Sym}^3(\pi)\in\mathfrak{A}_{4}$, and $\mathrm{Sym}^4(\pi)\in\mathfrak{A}_{5}$.
	\item If $\pi\in\mathfrak{F}_2$, then $\pi$ is dihedral if and only if $\mathrm{Sym}^2(\pi)\notin\mathfrak{F}_{3}$.  If $\pi$ is dihedral by $(\eta,\xi,K)$, then
	\begin{align*}
	A^2(\pi) &= I_K^F(\xi{\xi'}^{-1})\boxplus \eta,\\
	A^3(\pi)&=\begin{cases}
		\pi\otimes\mu\boxplus \pi\otimes\mu\eta &\parbox{0.4\linewidth}{if there exists $\mu\in\mathfrak{F}_1$ such that $\mu^2=\mathbbm{1}$ and $\xi\xi'^{-1}=\mu\circ\mathrm{N}_{K/F}$,}\\
		I_K^F(\xi^2\xi'^{-1})\boxplus\pi &\mbox{otherwise,}
	\end{cases}\\
	A^4(\pi)&=\begin{cases}
		\mathbbm{1}\boxplus		\mathbbm{1}\boxplus\mu\boxplus\eta\boxplus\mu\eta &\parbox{0.4\linewidth}{if there exists $\mu\in\mathfrak{F}_1$ such that $\mu^2=\mathbbm{1}$ and $\xi\xi'^{-1}=\mu\circ\mathrm{N}_{K/F}$,}\\
		\mathbbm{1}\boxplus I_K^F(\xi^3\xi'^{-1})\otimes\overline{\omega}_{\pi}\boxplus I_K^F(\xi^2)\otimes\overline{\omega}_{\pi}&\mbox{otherwise.}
	\end{cases}
	\end{align*}
	\item Let $\pi\in\mathfrak{F}_{2}$ be non-dihedral, in which case $A^2(\pi)\in\mathfrak{F}_{3}$.
	\begin{enumerate}
	\item $\Sym^3(\pi)\notin\mathfrak{F}_4$ if and only if there exists a non-trivial cubic character $\mu\in\mathfrak{F}_1$ such that $\pi$ is tetrahedral by $\mu$, in which case $A^2(\pi)=A^2(\pi)\otimes\mu$, $A^3(\pi)=\pi\otimes\mu \boxplus\pi\otimes\bar{\mu}$, and $A^4(\pi) = A^2(\pi)\boxplus \mu\boxplus\overline{\mu}$.
	\item If $\mathrm{Sym}^3(\pi)\in\mathfrak{F}_{4}$, then $\mathrm{Sym}^4(\pi)\notin\mathfrak{F}_{5}$ if and only if there exists a non-trivial quadratic character $\eta\in\mathfrak{F}_1$ such that $\pi$ is octahedral by $\eta$, in which case there exists a dihedral $\nu=\nu_{\pi}\in\mathfrak{F}_2$ such that $A^3(\pi)=A^3(\pi)\otimes\eta$ and $A^4(\pi) =  \nu \boxplus A^2(\pi)\otimes\eta$.
\end{enumerate}
\item If $\pi\in\mathfrak{F}_2$, then $\pi$ is not of solvable polyhedral type if and only if $A^2(\pi)\in\mathfrak{F}_3$, $A^3(\pi)\in\mathfrak{F}_4$, and $A^4(\pi)\in\mathfrak{F}_5$.
\end{enumerate}
\end{lemma}

Since $\Sym^m(\pi_v)$ is well-defined for each place $v$ of $F$ and each $m\geq 1$, as described in Section \ref{sec:intro}, the following lemma holds for all unramified places $v$ and all symmetric powers.

\begin{lemma}
\label{lem:CG}
If $j,k\geq 0$ are integers, $\pi\in\mathfrak{F}_{2}$ has central character $\omega$, $\chi\in\mathfrak{F}_1$, and $v\notin S_{\pi}^{\infty}\cup S_{\chi}^{\infty}$ is a place of $F$, then
\[
\mathrm{Sym}^{j}(\pi_v)\otimes\mathrm{Sym}^k(\pi_v)\otimes\chi_v=\bigoplus_{r=0}^{\min\{j,k\}}\mathrm{Sym}^{j+k-2r}(\pi_v)\otimes\omega_{v}^{r}\chi_v.
\]
In particular, if $j,k\in\{0,1,2\}$, then
\[
A^j(\pi_v)\otimes A^k(\widetilde{\pi}_v)\otimes\chi_v = \begin{cases}
\bigoplus_{r=0}^{\min\{j,k\}} A^{j+k-2r}(\pi_v)\otimes\chi_v&\mbox{if $(j,k)\notin\{(0,1),(2,1)\}$,}\vspace{1mm}\\
\bigoplus_{r=0}^{\min\{j,k\}} A^{j+k-2r}(\pi_v)\otimes \bar{\omega}_{\pi}\chi_v&\mbox{if $(j,k)\in\{(0,1),(2,1)\}$.}
\end{cases}
\]
\end{lemma}
\begin{proof}
	These follow from the Clebsch--Gordan identities.
\end{proof}

\begin{lemma}
\label{lem:AC}
If $\pi\in\mathfrak{F}_2$ and $j\leq 4$, then $\log C(A^j(\pi))\ll \log \mathfrak{C}_{\pi}$ and $\log \mathfrak{C}_{\omega_{\pi}}\ll \log \mathfrak{C}_{\pi}$.
\end{lemma}
\begin{proof}
There is nothing to prove when $j=1$.  Otherwise, the $L$-function identities
\begin{equation*}
\begin{gathered}
L(s,\pi\times\tilde{\pi})=\zeta_F(s)\cdot L(s,A^2(\pi)),\qquad L(s,A^2(\pi)\times\pi) = L(s,A^3(\pi))L(s,\pi),\\
L(s,A^2(\pi)\times A^2(\pi))=\zeta_F(s)\cdot L(s,A^2(\pi))\cdot L(s,A^4(\pi)).
\end{gathered}
\end{equation*}
yield the analytic conductor identities
\[
\mathfrak{C}_{A^2(\pi)}=\frac{\mathfrak{C}_{\pi\times\tilde{\pi}}}{\mathfrak{C}_{\mathbbm{1}}},\quad C(A^3(\pi))=\frac{C(A^2(\pi)\times\pi)}{\mathfrak{C}_{\pi}},\quad C(A^4(\pi))=\frac{C(A^2(\pi)\times A^2(\pi))}{\mathfrak{C}_{\mathbbm{1}}\mathfrak{C}_{A^2(\pi)}}.
\]
The claimed bounds now follow from \eqref{eqn:BH} and~\cite[Theorem A]{RamakrishnanYang}.
\end{proof}

\subsection{The symmetric square lift from $\mathrm{GL}_n$}

Let $(\Sigma,\chi)\in\mathfrak{F}_{n}\times\mathfrak{F}_{1}$.  Let $S$ be a set of places containing $S_{\Sigma}^{\infty}\cup S_{\chi}^{\infty}$ and all places dividing $2$.  The $\chi$-twist of the partial $L$-function of the symmetric square representation $\mathrm{Sym}^2\colon \mathrm{GL}_n(\mathbb{C})\to\mathrm{GL}_{n(n+1)/2}(\mathbb{C})$ is
\begin{equation}
\label{eqn:SYM2def}
L^S(s,\Sigma;\mathrm{Sym}^2\otimes\chi)=\prod_{v\notin S}~\prod_{1\leq j\leq k\leq n}\frac{1}{1-\chi_v(\varpi_v)\alpha_{j,\Sigma}(v)\alpha_{k,\Sigma}(v)q_v^{-s}}.
\end{equation}
\begin{theorem}[{\cite[Theorem 7.1]{Takeda}}]
\label{thm:sym2}
Let $\Sigma\in\mathfrak{F}_{n}$ and $\chi\in\mathfrak{F}_{1}$.  If $S$ is a set of places containing $S_{\Sigma}^{\infty}\cup S_{\chi}^{\infty}\cup\{v\mid 2\}$, then $L^S(s,\Sigma;\mathrm{Sym}^2\otimes\chi)$ is holomorphic on $\mathbb{C}$ except possibly for simple poles at $s\in\{0,1\}$.  If $\chi^n \omega_{\Sigma}^2\neq\mathbbm{1}$, then there is no pole.
\end{theorem}
Our next result uses Theorem~\ref{thm:sym2} in a situation where $L^S(s,\Sigma;\mathrm{Sym}^2\otimes\chi)$ is entire and $\chi^n\omega_{\Sigma}^2=\mathbbm{1}$.
\begin{corollary}
\label{cor:sym2}
If $\pi\in\mathfrak{F}_2$, $\chi\in\mathfrak{F}_1$, and $\pi$ is not of solvable polyhedral type, then
\begin{equation}
\label{eqn:sym2}
\frac{L(s,A^2(\pi)\times(A^4(\pi)\otimes\chi))}{L(s,A^4(\pi)\otimes\chi)}
\end{equation}
is holomorphic in the region $\re(s)>56/65$.
\end{corollary}
\begin{proof}
Since $\pi$ is not of solvable polyhedral type, $A^2(\pi)\in\mathfrak{F}_3$ and $A^4(\pi)\in\mathfrak{F}_5$.  Therefore, the numerator and denominator of \eqref{eqn:sym2} are each entire and non-vanishing in the half-plane $\re(s)\geq 1$.  It follows that \eqref{eqn:sym2} has no pole at $s=1$.

Proceeding locally, much like in~\cite[Lemma 7.1]{RamakrishnanWang}, we find that if $S$ is a set of places containing $S_{\pi}^{\infty}\cup S_{\chi}^{\infty}\cup\{v\mid 2\}$, then \eqref{eqn:sym2} equals
\begin{equation}
\label{eqn:factor_ratio_Sym4}
L^S(s,\Sym^3(\pi);\Sym^2\otimes\bar{\omega}_{\pi}^3\chi)\prod_{v\in S}\frac{L(s,A^2(\pi_v)\times(A^4(\pi_v)\otimes\chi_v))}{L(s,A^4(\pi_v)\otimes\chi_v)}.
\end{equation}
It follows from \eqref{eqn:Ramanujan1} and \eqref{eqn:Ramanujan2} that the product over $v\in S$ is holomorphic in the region
\[
\re(s)>\Big(\frac{1}{2}-\frac{1}{5^2+1}\Big)+\Big(\frac{1}{2}-\frac{1}{3^2+1}\Big)=\frac{56}{65}.
\]
Since \eqref{eqn:sym2} and the product over $v\in S$ on the right-hand side of \eqref{eqn:factor_ratio_Sym4} are each holomorphic at $s=1$, so is $L^S(s,\Sym^3(\pi);\Sym^2\otimes\bar{\omega}_{\pi}^3\chi)$.  Therefore, by Theorem~\ref{thm:sym2}, $L^S(s,\Sym^3(\pi);\Sym^2\otimes\bar{\omega}_{\pi}^3\chi)$ is holomorphic in the half-plane $\re(s)>1/2$.  We conclude that \eqref{eqn:sym2} is holomorphic in the half-plane $\re(s)>56/65$.
\end{proof}

\subsection{Results on modularity}

Let $\pi\in\mathfrak{F}_n$ and $\pi'\in\mathfrak{F}_{n'}$.  If $L(s,\pi\times\pi')$ is modular, then there exists a representation $\pi\boxtimes\pi'\in \mathfrak{A}_{n'n}$ such that $L(s,\pi\times\pi')=L(s,\pi\boxtimes\pi')$.

\begin{theorem}[\cite{KimShahidi,Ramakrishnan,RamakrishnanWang2}]
\label{thm:Langlands}
If $\pi\in\mathfrak{F}_2\cup\mathfrak{F}_3$ and $\pi'\in\mathfrak{F}_2$, then there exists an isobaric automorphic representation $\pi\boxtimes\pi'$ such that $L(s,\pi\times\pi')=L(s,\pi\boxtimes\pi')$.
\begin{enumerate}[leftmargin=*]
\item If $\pi,\pi'\in\mathfrak{F}_2$ are non-dihedral, then $\pi\boxtimes\pi'$ is cuspidal if and only if $\pi\not\sim\pi'$.
\item  If $\pi\in\mathfrak{F}_3$ and $\pi'\in\mathfrak{F}_2$, then $\pi\boxtimes\pi'$ is non-cuspidal if and only if
\begin{enumerate}
\item $\pi'$ is non-dihedral and $\pi\sim\Sym^2(\pi')$, or
\item $\pi'$ is dihedral, there exists an idele class character $\chi$ of a cubic non-normal extension $K/F$ such that $L(s,\pi)=L(s,\chi)$, and the base change $\pi_K$ is Eisensteinian.
\end{enumerate}
\item If $\pi,\pi' \in \fF_2$ are non-dihedral and $\pi \not\sim\pi'$, then $A^2(\pi) \not\sim A^2(\pi')$ and $\pi \boxtimes A^2(\pi')$ is cuspidal.
\end{enumerate}
\end{theorem}

\subsection{Other auxiliary results}

\begin{lemma}
\label{counting pole lemma}
Let $\pi,\pi' \in \fF_2$ be non-dihedral and $\chi \in \fF_1$. Let $\pi \not\sim \pi'$, $m,n\in\{1,2,3,4\}$, and $\mathfrak{o}_{m,n}(\pi,\pi')=-\mathrm{ord}_{s=1}L(s,A^m(\pi) \times (A^n(\pi') \otimes \chi))$.  Then $L(s,A^m(\pi) \times (A^n(\pi') \otimes \chi))$ is entire except possibly in the following cases:
\begin{enumerate}[leftmargin=*]
    \item $m=n=3$ and $\pi,\pi'$ are both non-tetrahedral, in which case $\mathfrak{o}_{m,n}(\pi,\pi')\in \{0,1\}$;
    \item $m=n=4$ and $\pi,\pi'$ are both tetrahedral, in which case $\mathfrak{o}_{m,n}(\pi,\pi')\in\{0,1,2,3,4\}$;
    \item $m=n=4$ and $\pi,\pi'$ are both octahedral, in which case $\mathfrak{o}_{m,n}(\pi,\pi')\in \{0,1\}$; and
    \item $m=n=4$ and $\pi,\pi'$ are both not of solvable polyhedral type, in which case $\mathfrak{o}_{m,n}(\pi,\pi')\in \{0,1\}$.
\end{enumerate}
\end{lemma}
\begin{proof}
    The proof is a direct application of Lemma~\ref{symmetric power decomposition} and Theorem~\ref{thm:Langlands}(3).  The proof is long, tedious, and unenlightening for our purposes, so we omit the details.
\end{proof}

\begin{lemma}
\label{lem:exercise}
If $\pi_0\in\mathfrak{F}_3$ is self-dual, then there exists a non-dihedral $\pi\in\mathfrak{F}_2$ such that $\pi_0=A^2(\pi)\otimes\omega_{\pi_0}$, and $\omega_{\pi_0}^2=\mathbbm{1}$.  If $\pi'\in\mathfrak{F}_2$ also satisfies $\pi_0=A^2(\pi')\otimes\omega_{\pi_0}$, then $\pi\sim\pi'$.
\end{lemma}
\begin{proof}
This is contained in~\cite[Theorem A]{Ramakrishnan_exercise}.
\end{proof}

We now elaborate on Proposition~\ref{prop:list}(6).

\begin{lemma}\cite[Theorems 1 and 2]{Luo}
\label{lem:Luo}
Let $\pi\in\mathfrak{F}_2$ and $\pi_0\in\mathfrak{F}_3$.  Let $\pi\boxtimes\pi_0$ be as in Theorem~\ref{thm:Langlands}.
\begin{enumerate}
	\item If $\pi\boxtimes\pi_0$ is cuspidal, then $L(s,\pi\times\pi_0)$ has no exceptional zero.
	\item If $\pi\boxtimes\pi_0$ is not cuspidal, then $L(s,\pi\times\pi_0)$ has at most one exceptional zero $\beta$ (necessarily simple).  If $\beta$ exists, then
	\begin{enumerate}
	\item $\pi$ is non-dihedral,
	\item there exists $\chi\in\mathfrak{F}_1$ such that $\pi_0=A^2(\pi)\otimes\chi$,
	\item $\beta$ is the sole (necessarily simple) exceptional zero of $L(s,A^3(\pi)\otimes\chi)$, and
	\item $\beta$ is a simple pole of $L(s,\pi,\Sym^5\otimes \bar{\omega}_{\pi}^2\chi)$.
	\end{enumerate}
\end{enumerate}
\end{lemma}

\section{Proofs of Theorems~\ref{symmetric fourth no Siegel zero}~and~\ref{thm:main3}}
\label{sec:proofs_Sym4+self-twist}

\begin{proof}[Proof of Theorem~\ref{symmetric fourth no Siegel zero}]

Let $(\pi,\chi)\in\mathfrak{F}_2\times\mathfrak{F}_1$.  We will study the zeros of $L(s,A^4(\pi)\otimes\chi)$.  To establish Theorem~\ref{symmetric fourth no Siegel zero}, we simply replace $\chi$ with $\chi\omega_{\pi}^2$.  We have three cases to consider.
\begin{enumerate}
    \item If $\pi$ is dihedral or tetrahedral, then by Lemma~\ref{symmetric power decomposition}(2,3a), $L(s,A^4(\pi) \otimes \chi)$ factors as a product of $L$-functions of automorphic representations of $\GL_m(\mathbb{A}_F)$ with $1 \leq m \leq 3$.  By Proposition~\ref{prop:list}(2), any exceptional zero of such a factor is a \textcolor{blue}{zero of a self-dual abelian factor}.
    \item If $\pi$ is octahedral by $\eta$, then Lemma~\ref{symmetric power decomposition}(3b) implies that there exists $\nu\in\mathfrak{F}_2$ such that $L(s,A^4(\pi) \otimes \chi)$ factors as $L(s,\nu \otimes \chi) \cdot L(s,A^2(\pi) \otimes \eta\chi)$, with $A^2(\pi)\in\mathfrak{F}_3$.  Therefore, by Proposition~\ref{prop:list}(2), each factor has \textcolor{red}{no exceptional zero}.
    \item If $\pi$ is not of solvable polyhedral type, then Lemma~\ref{symmetric power decomposition}(4) implies that $A^2(\pi)\in\mathfrak{F}_3$, $A^4(\pi)\in\mathfrak{F}_5$, and both are self-dual.
    \begin{enumerate}
    \item If $A^4(\pi)\otimes\chi\in\mathfrak{F}_5$ is not self-dual, then, by Proposition~\ref{prop:list}(1), $L(s,A^4(\pi)\otimes\chi)$ has \textcolor{red}{no exceptional zero}.
    \item If $A^4(\pi)\otimes\chi\in\mathfrak{F}_5$ is self-dual, then consider $\Pi = \mathbbm{1} \boxplus A^2(\pi) \boxplus (A^4(\pi) \otimes \chi)\in\mathfrak{A}_9$.  The $L$-function $L(s,\Pi\times\tilde{\Pi})$ factors as
    \[
    \zeta_F(s) L(s,A^2(\pi) \times A^2(\pi)) L(s,A^4(\pi) \times A^4(\pi)) L(s,A^4(\pi)\otimes\chi)^2 L(s,A^2(\pi) \times (A^4(\pi) \otimes \chi))^2 L(s,A^2(\pi))^2,
    \]
    which is holomorphic away from a pole of order $3$ at $s=1$.  By Corollary~\ref{cor:sym2}, any nontrivial zero $\rho$ of $L(s,A^4(\pi)\otimes\chi)$ with $\re(\rho)>56/65$ is a zero of $L(s,\Pi\times\tilde{\Pi})$ of multiplicity at least $4$.   Since $-(L'/L)(s,\Pi\times\tilde{\Pi})$ has nonnegative Dirichlet coefficients by Lemma~\ref{lem:HR}, the existence of an exceptional zero of $L(s,A^4(\pi)\otimes\chi)$ would contradict Lemma~\ref{lem:GHL} applied to $L(s,\Pi\times\tilde{\Pi})$.  Therefore, there is \textcolor{red}{no exceptional zero}.\qedhere
    \end{enumerate}
\end{enumerate}

\begin{proof}[Proof of Theorem~\ref{thm:main3}]
Let $(\pi,\pi')\in\mathfrak{F}_n\times\mathfrak{F}_{n'}$.  Suppose that $\psi\in\mathfrak{F}_1$ satisfies $\pi\otimes\psi=\pi$ and $\pi'\otimes\psi\neq \pi'$.  If $\Pi = \widetilde{\pi} \boxplus \pi' \boxplus (\pi' \otimes \psi)\in\mathfrak{A}_{n+2n'}$, then our hypotheses imply that
\[
L(s,\Pi\times\tilde{\Pi}) = L(s,\pi\times\pi')^2 \cdot L(s,\tilde{\pi}\times\tilde{\pi}')^2  \cdot L(s,\pi\times\tilde{\pi})\cdot L(s,\pi'\times\tilde{\pi}')^2 \cdot L(s,\pi'\times(\tilde{\pi}'\otimes\psi)) \cdot L(s,\pi'\times(\tilde{\pi}'\otimes\bar{\psi})).
\]
Since $\pi' \neq \pi' \otimes \psi$, the product $L(s,\pi'\times(\tilde{\pi}'\otimes\psi))\cdot L(s,\pi'\times(\tilde{\pi}'\otimes\bar{\psi}))$ is entire.  Therefore, $L(s,\Pi\times\tilde{\Pi})$ is holomorphic away from a pole of order $3$ at $s=1$.  By Lemma~\ref{lem:HR}, $-(L'/L)(s,\Pi\times\tilde{\Pi})$ has nonnegative Dirichlet coefficients.  Since any real zero of $L(s,\pi\times\pi')$ is also a real zero of $L(s,\tilde{\pi}\times\tilde{\pi}')$ by \eqref{eqn:dual}, the existence of an exceptional zero of $L(s,\pi\times\pi')$ would contradict Lemma~\ref{lem:GHL} applied to $L(s,\Pi\times\tilde{\Pi})$.  Therefore, $L(s,\pi\times\pi')$ has \textcolor{red}{no exceptional zero}.
    \end{proof}
	
\end{proof}

\section{Proof strategy for Theorem~\ref{thm:main}}
\label{sec:strategy}

The following proposition is the only known strategy for eliminating exceptional zeros.

\begin{proposition}
\label{prop:strategy}
Let $(\pi_1,\pi_2)\in\mathfrak{F}_{n_1}\times\mathfrak{F}_{n_2}$.  If there exist a product $D(s)$ of Rankin--Selberg $L$-functions, integers $\ell_1,\ell_2\geq 0$ and $k\geq 1$ satisfying $\ell_1+\ell_2>k$, and a fixed $t\in(0,1)$ such that
\begin{enumerate}
	\item at each unramified $v\nmid\infty$ and $\ell\geq 1$, we have $\mathrm{Re}(a_D(v^{\ell}))\geq 0$ (see \eqref{eqn:aDdef}),
	\item $D(s)$ is holomorphic everywhere except for a pole of order $k$ at $s=1$,
	\item $D(s) L(s,\pi_1\times\pi_2)^{-\ell_1}L(s,\widetilde{\pi}_1\times\widetilde{\pi}_2)^{-\ell_2}$ is holomorphic at each real $s\in(t,1)$,
    \item the degree of the Euler product defining $D(s)$ for $\re(s)>1$ is $d\geq 1$, and
    \item the logarithm of the analytic conductor $\mathfrak{C}_D$ of $D(s)$ is $O_{n_1,n_2}(\log(\mathfrak{C}_{\pi_1}\mathfrak{C}_{\pi_2}))$,
\end{enumerate}
then there exists an effectively computable constant $\Cl[abcon]{ZFR_GHL}=\Cr{ZFR_GHL}(d)>0$ such that $L(\sigma,\pi_1\times\pi_2)$ has no real zero in the interval $\sigma\geq 1-\Cr{ZFR_GHL}/\log(\mathfrak{C}_{\pi_1}\mathfrak{C}_{\pi_2})$.
\end{proposition}
\begin{proof}
By \eqref{eqn:dual}, a real zero of $L(s,\pi_1\times\pi_2)$ is a zero of $D(s)$ with multiplicity at least $\ell_1+\ell_2>k$, so the existence of an exceptional zero of $L(s,\pi_1\times\pi_2)$ contradicts Lemma~\ref{lem:GHL} applied to $D(s)$.
\end{proof}

\subsection{Earlier work}
\label{subsec:earlier_work}

First, we consider the case where $\pi_1\in\cup_{n=2}^{\infty}\mathfrak{F}_{n}$ and $\pi_2=\mathbbm{1}$.  By Proposition~\ref{prop:list}(1), it suffices to assume that $\pi_1=\tilde{\pi}_1$.  Hoffstein and Ramakrishnan~\cite[Proof~of~Theorem~B]{HoffsteinRamakrishnan} showed that sufficient progress towards the modularity of Rankin--Selberg $L$-functions suffices to prove that $L(s,\pi_1)$ has no exceptional zero.  Assume that $L(s,\pi_1\times\widetilde{\pi}_1)$ is modular, so that there exists a representation $\pi_1\boxtimes\widetilde{\pi}_1\in\mathfrak{A}_{n^2}$ satisfying $L(s,\pi_1\times\widetilde{\pi}_1)=L(s,\pi_1\boxtimes\widetilde{\pi}_1)$.  By~\cite[Lemma~4.4]{HoffsteinRamakrishnan}, $\pi_1\boxtimes\widetilde{\pi}_1$ has a cuspidal constituent $\tau\notin\{\mathbbm{1},\pi_1,\widetilde{\pi}_1\}$.  If $L(s,\pi_1\times\tau)$ is modular, then by~\cite[Proofs~of~Lemma~4.4~and~Claim~4.5]{HoffsteinLockhart}, $L(s,\pi_1\times\tau)/L(s,\pi_1)$ is entire.  Therefore, subject to the modularity of $L(s,\pi_1\times\widetilde{\pi}_1)$ and $L(s,\pi_1\times\tau)$, Hoffstein and Ramakrishnan prove that if $\Pi=\mathbbm{1}\boxplus\widetilde{\tau}\boxplus\pi_1$, then $D(s)=L(s,\Pi\times\widetilde{\Pi})$ satisfies the hypotheses of Proposition~\ref{prop:strategy} with $k=3$ and $\ell_1=\ell_2=2$.

Proposition~\ref{prop:list}(2,4,5,6) are proved by verifying the existence of $\pi_1\boxtimes\widetilde{\pi}_1$ and the existence of a nontrivial cuspidal constituent $\tau$ of $\pi_1\boxtimes\widetilde{\pi}_1$ such that there exists an effectively computable constant $\Cl[abcon]{HR}=\Cr{HR}(n,F)\in(0,1)$ satisfying the property that $L(\sigma,\pi_1\times\tau)/L(\sigma,\pi_1)$ is holomorphic at each $\sigma\in(\Cr{HR},1]$.  Proposition~\ref{prop:list}(1,3,7,8) is proved by constructing a $D(s)$ satisfying Proposition~\ref{prop:strategy} using self-duality or twist-equivalence hypotheses.

In~\cite{Thorner_Siegel}, Proposition~\ref{prop:Thorner} is proved in a manner equivalent to the description we now give.  Suppose that $F$ is totally real, $\pi,\pi'\in\mathfrak{F}_2$ are non-dihedral and regular algebraic (so that $\Sym^n(\pi),\Sym^n(\pi')\in\mathfrak{F}_{n+1}$ for all $n\geq 1$~\cite{NewtonThorne3}), and $\pi\not\sim\pi'$.  Consider
\begin{equation}
\label{eqn:JT_Sym}
\begin{aligned}
\mathcal{D}(s)&=L(s,\mathrm{Sym}^{m}(\pi)\times\mathrm{Sym}^m(\widetilde{\pi})) L(s,\mathrm{Sym}^{n}(\pi')\times\mathrm{Sym}^n(\widetilde{\pi}'))^2 L(s,\mathrm{Sym}^4(\pi)\otimes\overline{\omega}_{\pi}^{2})\\
&\cdot L(s,\mathrm{Sym}^{m}(\pi)\times(\mathrm{Sym}^{n}(\pi')\otimes\chi))^2 L(s,\mathrm{Sym}^m(\widetilde{\pi})\times(\mathrm{Sym}^n(\widetilde{\pi}')\otimes\overline{\chi}))^2 L(s,\mathrm{Ad}(\pi))^3\\
&\cdot L(s,\mathrm{Sym}^{m-2}(\pi)\times(\mathrm{Sym}^{n}(\pi')\otimes\chi\omega_{\pi})) L(s,\mathrm{Sym}^{m-2}(\widetilde{\pi})\times(\mathrm{Sym}^{n}(\widetilde{\pi}')\otimes\overline{\chi}\overline{\omega}_{\pi})) \\
&\cdot L(s,\mathrm{Sym}^{m+2}(\pi)\times(\mathrm{Sym}^{n}(\pi')\otimes\chi\overline{\omega}_{\pi})) L(s,\mathrm{Sym}^{m+2}(\widetilde{\pi})\times(\mathrm{Sym}^n(\widetilde{\pi}')\otimes\overline{\chi}\omega_{\pi}))\\
&\cdot \prod_{k=1}^n [L(s,\mathrm{Ad}(\pi)\times(\mathrm{Sym}^{2k}(\pi')\otimes\overline{\omega}_{\pi'}^{k}))^3 L(s,\mathrm{Sym}^4(\pi)\times(\mathrm{Sym}^{2k}(\pi')\otimes \overline{\omega}_{\pi}^{2}\overline{\omega}_{\pi'}^{k}))].
\end{aligned}
\end{equation}
For a place $v$ of $F$, we define
\[
\Pi_v = \Sym^m(\tilde{\pi}_v)\oplus A^2(\pi_v)\otimes\Sym^n(\pi_v')\otimes\chi_v\oplus \Sym^n(\pi_v')\otimes\chi_v.
\]
Note that $\bigotimes_v \Pi_v$ is not yet known to lie in $\mathfrak{A}_{m+4n+5}$ because $\bigotimes_v A^2(\pi_v)\otimes\Sym^n(\pi_v')\otimes\chi_v$ is not yet known to lie in $\mathfrak{A}_{3(n+1)}$ except when $n\in\{0,1\}$.  However, it follows from Lemma~\ref{lem:CG} that if $v\notin S_{\pi}^{\infty}\cup S_{\pi'}^{\infty}\cup S_{\chi}^{\infty}$, then the $v$-th Euler factor of $\mathcal{D}(s)$ is $L(s,\Pi_v\otimes\tilde{\Pi}_v)$, which satisfies
\[
-\frac{L'}{L}(s,\Pi_v\otimes\tilde{\Pi}_v)=\sum_{\ell=1}^{\infty}\frac{|\overline{a_{\Sym^m(\pi)}(v^{\ell})}+a_{A^2(\pi)}(v^{\ell})a_{\Sym^n(\pi')\otimes\chi}(v^{\ell})+a_{\Sym^n(\pi')\otimes\chi}(v^{\ell})|^2\log q_v}{q_v^{\ell s}}.
\]
It follows that for such $v$, the $v^{\ell}$-th Dirichlet coefficient of $-(\mathcal{D}'/\mathcal{D})(s)$ is nonnegative.

If $m=2$ and $n=0$, then $L(s,\mathrm{Sym}^{m}(\pi)\times(\mathrm{Sym}^n(\pi')\otimes\chi)) = L(s,\Sym^2(\pi)\otimes\chi)$, which has no exceptional zero by Proposition~\ref{prop:list}(2).  Otherwise, the poles of $\mathcal{D}(s)$ should only come from $L(s,\mathrm{Sym}^{m}(\pi)\times\mathrm{Sym}^m(\widetilde{\pi})) L(s,\mathrm{Sym}^{n}(\pi')\times\mathrm{Sym}^n(\widetilde{\pi}'))^2$.  The only factors at risk of introducing another pole are of the form $L(s,\Sym^j(\pi)\times(\Sym^j(\pi')\otimes\xi))$ with $\xi\in\mathfrak{F}_1$, so it remains to show that such $L$-functions are entire when $\pi\not\sim\pi'$.  Rankin--Selberg theory and progress towards modularity are currently insufficient to conclude this expectation on their own.  However, since $\pi$ and $\pi'$ are regular algebraic, if there exists $t_{\xi}\in\mathbb{R}$ such that $\xi|\cdot|^{it_{\xi}}$ corresponds with a ray class character over $F$, then one can use the $\ell$-adic representations associated to $\pi$ and $\pi'$ along with the Chebotarev density theorem to establish the absence of extra poles~\cite[Lemma~4.2]{Thorner_Siegel}.  Restricting to $\chi$ satisfying this hypothesis, and observing that
\[
L(s,\mathrm{Sym}^{m}(\pi)\times(\mathrm{Sym}^{n}(\pi')\otimes\chi))^2 L(s,\mathrm{Sym}^m(\widetilde{\pi})\times(\mathrm{Sym}^n(\widetilde{\pi}')\otimes\overline{\chi}))^2
\]
is a factor of $\mathcal{D}(s)$, we see that any real zero $\beta$ of $L(s,\mathrm{Sym}^{m}(\pi)\times(\mathrm{Sym}^{n}(\pi')\otimes\chi))$ is a zero of $\mathcal{D}(s)$ with order at least 4, thus eliminating the existence of exceptional zeros via Proposition~\ref{prop:strategy}.

\subsection{Our approach}
\label{subsec:OurApproach}

In contrast with~\cite{Thorner_Siegel}, we must contend with two significant limitations.  First, we can only use information from the $\GL_1$-twists of the first four symmetric powers from $\GL_2$ and the symmetric square lift of $\Sym^3$.  Second, our information on the possible poles of the $\GL_1$-twists of $L$-functions of the form $L(s,A^j(\pi)\times A^j(\pi'))$ is limited to Lemma~\ref{counting pole lemma}.  We must therefore construct auxiliary Dirichlet series that use the available information more efficiently than in previous works.  We do so as follows.  For simplicity of notation, we will focus our discussion in this section on Theorem~\ref{thm:main}(1a,1b,1c,2,3).

\begin{lemma}
\label{lem:D}
Let $\chi\in\mathfrak{F}_1$.  Let $\pi,\pi'\in\mathfrak{F}_2$ be non-dihedral and twist-inequivalent.  Given $j,k\in\{0,1,2\}$ and $r\in\{0,1\}$, let $c_{j,k,r}\in\mathbb{N}\cup\{0\}$.  Define $D(s)$ by
\[
\prod_{\substack{j,j'\in\{0,1,2\} \\ k,k'\in\{0,1,2\} \\ r,r'\in\{0,1\}}}\prod_{a=0}^{\min\{j,j'\}}\prod_{b=0}^{\min\{k,k'\}}L(s, A^{j+j'-2a}(\pi)\times (A^{k+k'-2b}(\pi')\otimes \bar{\omega}_{\pi}^{\mathbf{1}_{2\mid j}\mathbf{1}_{j'=1}}\bar{\omega}_{\pi'}^{\mathbf{1}_{2\mid k}\mathbf{1}_{k'=1}}\chi^{r-r'}))^{c_{j,k,r}c_{j',k',r'}}.
\]
The logarithm of the analytic conductor of $D(s)$ is $O(\log(\mathfrak{C}_{\pi}\mathfrak{C}_{\pi'}\mathfrak{C}_{\chi}))$.  Also, if
\[
\sum_{v\nmid\infty}\sum_{\ell=1}^{\infty}\frac{a_D(v^{\ell})\log q_v}{q_v^{\ell s}},\qquad \re(s)>1
\]
is the Dirichlet series expansion of $-(D'/D)(s)$ and $v\notin S_{\pi}^{\infty}\cup S_{\pi'}^{\infty}\cup S_{\chi}^{\infty}$, then $a_D(v^{\ell})\geq 0$.
\end{lemma}
\begin{proof}
The first result follows from Lemma~\ref{lem:AC}.  For a place $v\notin S_{\pi}^{\infty}\cup S_{\pi'}^{\infty}\cup S_{\chi}^{\infty}$, define
\begin{equation}
\label{eqn:D_def_v}
\Pi_v=\bigoplus_{j=0}^2 \bigoplus_{k=0}^2 \bigoplus_{r=0}^1 c_{j,k,r} A^j(\pi_v)\otimes A^k(\pi_v')\otimes \chi_v^{r}.
\end{equation}
On one hand, by Lemma~\ref{lem:CG}, $\Pi_v\otimes\tilde{\Pi}_v$ decomposes as
\begin{align*}
&\bigoplus_{\substack{j,j'\in\{0,1,2\} \\ k,k'\in\{0,1,2\} \\ r,r'\in\{0,1\}}} c_{j,k,r}c_{j',k',r'} (A^j(\pi_v)\otimes A^{j'}(\widetilde{\pi}_v)) \otimes(A^k(\pi_v')\otimes A^{k'}(\widetilde{\pi}_v'))\otimes \chi_v^{r-r'}\\
%&=\bigoplus_{\substack{j,j'\in\{0,2\} \\ k,k'\in\{0,2\} \\ r,r'\in\{0,1\}}} c_{j,k,r}c_{j',k',r'} (A^j(\pi_v)\otimes A^{j'}(\pi_v)\otimes\bar{\omega}_{\pi}^{\mathbf{1}_{j\mid 2}\mathbf{1}_{j'=1}}) \otimes(A^k(\pi_v')\otimes A^{k'}(\pi_v')\otimes\bar{\omega}_{\pi'}^{\mathbf{1}_{k\mid 2}\mathbf{1}_{k'=1}})\otimes \chi^{r-r'}_v\\
%&=\bigoplus_{\substack{j,j'\in\{0,2\} \\ k,k'\in\{0,2\} \\ r,r'\in\{0,1\}}} c_{j,k,r}c_{j',k',r'} [\Sym^j(\pi)\boxtimes\Sym^{j'}(\pi)]_v \otimes[\Sym^k(\pi')\otimes\Sym^{k'}(\pi')]_{v}\otimes (\bar{\omega}_{\pi}^{j'}\bar{\omega}_{\pi'}^{k'}\chi^{r-r'})_v\\
&=\bigoplus_{\substack{j,j'\in\{0,1,2\} \\ k,k'\in\{0,1,2\} \\ r,r'\in\{0,1\}}} c_{j,k,r}c_{j',k',r'} \bigoplus_{a=0}^{\min\{j,j'\}} \bigoplus_{b=0}^{\min\{k,k'\}} A^{j+j'-2a}(\pi_v)\otimes A^{k+k'-2b}(\pi_v')\otimes (\bar{\omega}_{\pi}^{\mathbf{1}_{2\mid j}\mathbf{1}_{j'=1}}\bar{\omega}_{\pi'}^{\mathbf{1}_{2\mid k}\mathbf{1}_{k'=1}}\chi^{r-r'})_v.
\end{align*}
This indicates that the $v$-th Euler factor of $D(s)$ is $L(s,\Pi_v\otimes\tilde{\Pi}_v)$.  On the other hand, one computes
\[
-\frac{L'}{L}(s,\Pi_v\otimes\tilde{\Pi}_v)=\sum_{\ell=1}^{\infty}\Big|\sum_{j=0}^2 \sum_{k=0}^2 \sum_{r=0}^1 c_{j,k,r}a_{A^j(\pi)}(v^{\ell})a_{A^k(\pi')}(v^{\ell}) a_{\chi}(v^{\ell})^{r}\Big|^2\frac{\log q_v}{q_v^{\ell s}}.
\]
The desired result follows.
\end{proof}

As of now, if $\Pi_v$ is given by \eqref{eqn:D_def_v}, then $\Pi=\bigotimes_v \Pi_v$ is only known to be a global isobaric automorphic representation when each coefficient $c_{j,k,r}$ satisfying $j+k=4$ equals zero.  Central to our proofs, Lemma~\ref{lem:D} shows that $D(s)$ is a product of Rankin--Selberg $L$-functions that satisfies the condition (1) in Proposition~\ref{prop:strategy}, even if there are coefficients $c_{j,k,r}\geq 1$ satisfying $j+k=4$.  Therefore, it remains to find suitable coefficients $c_{j,k,r}$ for which $D(s)$ satisfies the remaining hypotheses of Proposition~\ref{prop:strategy} for precluding the exceptional zeros of the $L$-functions in Theorem~\ref{thm:main}(1a,1b,1c,2,3).

We also address the cases where $\pi$ or $\pi'$ is dihedral, or $\pi\sim\pi'$, or $\Sym^2(\pi')$ is replaced with an arbitrary $\pi_0\in\mathfrak{F}_3$.  These lead to extensive casework that we handle using Proposition~\ref{prop:list}, Lemmata~\ref{symmetric power decomposition}~and~\ref{lem:CG}, and Theorem~\ref{thm:Langlands}.

Our approach to eliminating exceptional zeros for the triple product $L$-functions in Theorem~\ref{thm:main}(1d,1e,1f) is only a minor modification of the strategy presented above.  It is notationally cumbersome to give a unified treatment like we did for Theorem~\ref{thm:main}(1a,1b,1c,2,3).

\subsection{Proof organization}

Let $m,n,r\geq 1$ be integers, $\chi \in \fF_1$, $\pi,\pi' \in \fF_2$, and $\pi_0\in\mathfrak{A}_3$.  We say that $L(s,A^m(\pi)\times (A^n(\pi')\otimes\chi))$ is in the
\begin{enumerate}
	\item {\it twist-inequivalent general case} if $\pi \not\sim \pi'$, and $A^m(\pi)$ and $A^n(\pi')$ are both cuspidal;
	\item {\it twist-inequivalent reduced case} if $\pi \not\sim \pi'$, and $A^m(\pi)$ or $A^n(\pi')$ is non-cuspidal;
	\item {\it twist-equivalent case} if $\pi \sim \pi'$.
\end{enumerate}
We say that $L(s,A^m(\pi)\times A^n(\pi')\times(\pi_0\otimes\chi))$ is in the
\begin{enumerate}
	\item {\it twist-inequivalent general case} if $\pi \not\sim \pi'$, $A^m(\pi)\not\sim\pi_0$,  $A^n(\pi')\not\sim\pi_0$, and $A^m(\pi)$ and $A^n(\pi')$ and $\pi_0$ are cuspidal;
	\item {\it twist-inequivalent reduced case} if $\pi \not\sim \pi'$, $A^m(\pi)\not\sim\pi_0$,  $A^n(\pi')\not\sim\pi_0$, and $A^m(\pi)$ or $A^n(\pi')$ or $\pi_0$ is non-cuspidal;
	\item {\it twist-equivalent case} if $\pi \sim \pi'$, $A^m(\pi)\sim\pi_0$, or $A^n(\pi')\sim\pi_0$.
\end{enumerate}
Typically, we will first proceed with the assumption that each $L$-function under consideration is in the twist-inequivalent general case.  We will then handle the twist-inequivalent reduced cases, and finally the twist-equivalent cases.

\section{Proof of Theorem~\ref{thm:main}(1a)}
\label{sec:1a}

Let $\pi \in \fF_2$ and $\pi_0 \in \fF_3$.  In this section, we will prove the following result.

\begin{theorem}
\label{thm:GL3xGL3}
If $\pi \in \fF_2$ and $\pi_0 \in \fF_3$, then any exceptional zero of $L(s,A^2(\pi)\times\pi_0)$ is the zero of a self-dual abelian factor.  No such factor exists unless $\pi_0\sim A^2(\pi)$.
\end{theorem}
\begin{proof}[Proof of Theorem~\ref{thm:main}(1a)]
In Theorem~\ref{thm:GL3xGL3}, replace $\pi_0$ with $\pi_0\otimes\omega_{\pi}$.
\end{proof}

\subsection{The twist-inequivalent general case}
Assume $\pi$ is non-dihedral and $\pi_0 \not\sim A^2(\pi)$.  It follows from Lemma~\ref{symmetric power decomposition}(2) that $A^2(\pi)\in\mathfrak{F}_3$.

\subsubsection{$\pi_0$ not self-dual}

Since $A^2(\pi) \in \fF_3$ is self-dual, it follows from Proposition~\ref{prop:list}(8) that if $\pi_0$ is not self-dual, then $L(s,A^2(\pi) \times \pi_0)$ has \textcolor{red}{no exceptional zero}.

\subsubsection{$\pi_0$ self-dual}
\label{subsubsec:3x3_SD}

If $\pi_0$ is self-dual, then by Lemma~\ref{lem:exercise}, there exists $\chi \in \fF_1$ with $\chi^2=\mathbbm{1}$ and a non-dihedral $\pi' \in \fF_2$ such that $\pi_0 = A^2(\pi') \otimes \chi$.  We now define
\begin{equation}
\label{eqn:D_1+D_2}
\begin{aligned}
    D_1(s) &= \zeta_F(s)^6 \cdot L(s,A^4(\pi) \times A^4(\pi')), \\
    D_2(s) &= L(s,A^2(\pi))^7\cdot  L(s,A^4(\pi))^5\cdot L(s,A^2(\pi'))^2\cdot  L(s,A^4(\pi'))^2    \cdot  L(s,A^2(\pi') \otimes \chi)^4 \\
    &  \cdot L(s,A^4(\pi) \times (A^2(\pi') \otimes \chi))^4\cdot  L(s,A^2(\pi) \times A^4(\pi'))^3  \cdot  L(s,A^2(\pi) \times A^2(\pi'))^3  \cdot   \\
    & \cdot L(s,A^4(\pi) \times A^2(\pi')),\\
    D(s)&= L(s,A^2(\pi) \times (A^2(\pi') \otimes \chi))^8 \cdot D_1(s) \cdot D_2(s).
\end{aligned}
\end{equation}
The $D(s)$ in \eqref{eqn:D_1+D_2} matches $D(s)$ in Lemma~\ref{lem:D} with $c_{2,0,0} = 2$, $c_{0,2,1} = c_{2,2,1} = 1$, and the remaining $c_{j,k,r}$'s equal to zero.  Our assumption that $\pi_0 \not\sim A^2(\pi)$ implies that $\pi \not\sim \pi'$. Thus, by Lemma~\ref{counting pole lemma}, $D_2(s)$ is entire, and $D_1(s)$ has a pole at $s=1$ of order at least 6.  Regarding the order of this pole, we discuss the following two cases.

{\it Case \ref{subsubsec:3x3_SD}a:  At least one of $\pi$ and $\pi'$ is non-tetrahedral.}  In this case, Lemma~\ref{counting pole lemma} states that $L(s,A^4(\pi) \times A^4(\pi'))$ (resp. $D_1(s)$) has at most a simple pole (resp. a pole of order at most $7$) at $s=1$.  Therefore, $L(s,A^2(\pi) \times \pi_0) = L(s,A^2(\pi) \times A^2(\pi') \otimes \chi)$ has \textcolor{red}{no exceptional zero} by Proposition~\ref{prop:strategy} applied to $D(s)$ in \eqref{eqn:D_1+D_2}, $\ell_1 = 8$, $\ell_2=0$, and $k \leq 7$. 

{\it Case \ref{subsubsec:3x3_SD}b:  $\pi$ and $\pi'$ are tetrahedral.} By Lemma~\ref{counting pole lemma}(2), $L(s,A^4(\pi) \times A^4(\pi'))$ has a pole at $s=1$ of order at most 4. By Lemma~\ref{symmetric power decomposition}(3a), if $\pi$ is tetrahedral by $\mu$, then
\[
L(s,A^4(\pi) \times (A^2(\pi') \otimes \chi)) = L(s,A^2(\pi) \times (A^2(\pi') \otimes \chi)) \cdot L(s,A^2(\pi') \otimes \chi\mu) \cdot L(s,A^2(\pi') \otimes \chi\mu^2).
\]
Inserting this factorization into \eqref{eqn:D_1+D_2}, we conclude that if $D_3(s)$ equals
\begin{align*}
&L(s,A^2(\pi))^7\cdot  L(s,A^4(\pi))^5\cdot L(s,A^2(\pi'))^2\cdot  L(s,A^4(\pi'))^2    \cdot  L(s,A^2(\pi') \otimes \chi)^4 \cdot L(s,A^2(\pi') \otimes \chi\mu)^4\\
& \cdot L(s,A^2(\pi') \otimes \chi\mu^2)^4 \cdot  L(s,A^2(\pi) \times A^4(\pi'))^3  \cdot  L(s,A^2(\pi) \times A^2(\pi'))^3  \cdot L(s,A^4(\pi) \times A^2(\pi'))
\end{align*}
then $D(s)$ in \eqref{eqn:D_1+D_2} factors as
\begin{equation}
\label{eqn:GL3xGL3_2}
D(s) = L(s,A^2(\pi) \times A^2(\pi') \otimes \chi)^{12} \cdot D_1(s) \cdot D_3(s).
\end{equation}
By Lemma~\ref{counting pole lemma}, $D_3(s)$ is entire, while $D_1(s)$ has a pole at $s=1$ of order at most $10$. In this case, $L(s,A^2(\pi) \times \pi_0) = L(s,A^2(\pi) \times A^2(\pi') \otimes \chi)$ has \textcolor{red}{no exceptional zero} by Proposition~\ref{prop:strategy} applied to $D(s)$ in \eqref{eqn:GL3xGL3_2}, with $\ell_1 = 12$, $\ell_2 = 0$, and $k \leq 10$.

\subsection{The twist-inequivalent reduced case}
In this subsection, we continue to assume that $\pi_0 \not\sim A^2(\pi)$, but we now assume that $\pi$ is dihedral by $(\eta,\xi,K)$.  By Lemma~\ref{symmetric power decomposition}(2) we have 
$$L(s,A^2(\pi) \times \pi_0) = L(s,I_K^F(\xi{\xi'}^{-1}) \times \pi_0) \cdot L(s,\pi_0 \otimes \eta),$$
which has \textcolor{red}{no exceptional zero} by Proposition~\ref{prop:list}(2,6).

\subsection{The twist-equivalent case} 
In this subsection, we assume that there exists $\chi\in\fF_1$ such that $\pi_0 = A^2(\pi) \otimes \chi$.  By Lemma~\ref{lem:CG}, we have that
\[
    L(s,A^2(\pi) \times \pi_0) = L(s,A^2(\pi) \times (A^2(\pi) \otimes \chi)) = L(s,A^4(\pi) \otimes \chi) \cdot L(s,A^2(\pi) \otimes \chi) \cdot L(s,\chi).
\]
Thus, any exceptional zero of $L(s,A^2(\pi) \times \pi_0)$ is the \textcolor{blue}{zero of a self-dual abelian factor} by Proposition~\ref{prop:list}(1,2) and Theorem~\ref{symmetric fourth no Siegel zero}.

\section{Proof of Theorem~\ref{thm:main}(1b)}
\label{sec:1b}

Let $\pi,\pi' \in \fF_2$ with $\pi'$ non-dihedral and $\pi \not\sim \pi'$. Let $\chi \in \fF_1$.  In this section, we will prove the following result.

\begin{theorem}
\label{thm:GL4xGL3}
Let $\pi,\pi' \in \fF_2$ and $\chi\in\mathfrak{F}_1$.  If $\pi'$ is non-dihedral and $\pi\not\sim\pi'$, then $L(s,A^3(\pi)\times(A^2(\pi')\otimes\chi))$ has no exceptional zero.
\end{theorem}

\begin{remark}
\label{rem:GL4xGL3}
If $\pi\sim\pi'$, or $\pi\not\sim\pi'$ but $\pi'$ is dihedral, then $L(s,\mathrm{Sym}^3(\pi)\times(\Sym^2(\pi')\otimes\chi))$ will have a factor that is a $\GL_1$-twist of $L(s,A^3(\pi))$ or $L(s,\pi,\Sym^5)$.  Lemma~\ref{lem:Luo} is now germane.  Outside of the setting of Proposition~\ref{prop:Thorner}, it is not yet known whether $L(s,\pi,\Sym^5\otimes\psi)$ is holomorphic in a real interval of the form $[1-\Cl[abcon]{sym5}/\log(\mathfrak{C}_{\pi}\mathfrak{C}_{\psi}),\infty)$.
\end{remark}

\begin{proof}[Proof of Theorem~\ref{thm:main}(1b)]
In Theorem~\ref{thm:GL4xGL3}, replace $\chi$ with $\chi \omega_{\pi}\omega_{\pi'}$.
\end{proof}

\subsection{The twist-inequivalent general case}
Assume that $A^3(\pi)$ and $A^2(\pi')$ are cuspidal, hence $A^3(\pi)\otimes\chi$ is cuspidal.  We separately handle the cases when $A^3(\pi)\otimes\chi$ is self-dual or not.
\subsubsection{$A^3(\pi)\otimes\chi$ is not self-dual}
It follows from the cuspidality and self-duality of $A^2(\pi')$ that $L(s,A^3(\pi)\times(A^2(\pi')\otimes\chi))$ has {\color{red}no exceptional zero} by Proposition~\ref{prop:list}(8).
\subsubsection{$A^3(\pi)\otimes\chi$ is self-dual}  In this case, we define
\begin{equation}
\label{eqn:GL4xGL3}
\begin{aligned}
    D_6(s) &= \zeta_F(s)^6 \cdot L(s,A^4(\pi) \times A^4(\pi')), \\
    D_7(s) &= L(s,\pi\otimes\chi)^2 \cdot L(s,\tilde{\pi}\otimes\bar{\chi})^2 \cdot  L(s,A^2(\pi))^6 \cdot  L(s,A^3(\pi)\otimes\chi)^4 \cdot  L(s,A^4(\pi))^2 \\
    &\cdot  L(s,A^2(\pi'))^7 \cdot L(s,A^4(\pi'))^5 \cdot  L(s,\pi\times(A^2(\pi')\otimes\chi))^4\cdot L(s,\tilde{\pi}\times(A^2(\pi')\otimes\bar{\chi}))^4 \\
    &\cdot L(s,\pi\times(A^4(\pi')\otimes\chi))^2\cdot L(s,\tilde{\pi}\times(A^4(\pi')\otimes\bar{\chi}))^2 \cdot L(s,A^2(\pi)\times A^2(\pi'))^7 \\
    &\cdot L(s,A^2(\pi)\times A^4(\pi'))^5\cdot L(s,A^3(\pi)\times(A^4(\pi')\otimes\chi))^4 \cdot L(s,A^4(\pi)\times A^2(\pi'))^3,\\
    D(s) & = L(s,A^3(\pi) \times(A^2(\pi')\otimes\chi))^8 \cdot D_6(s) \cdot D_7(s).
\end{aligned}
\end{equation}
The $D(s)$ in \eqref{eqn:GL4xGL3} matches $D(s)$ from Lemma~\ref{lem:D} with $c_{2,0,0} = 1$,  $c_{1,2,1} = 2$, $c_{2,2,0} = 1$, and $c_{j,k,r}=0$ otherwise.  By Lemma~\ref{counting pole lemma}, $D_7(s)$ is entire, and $L(s,A^4(\pi) \times A^4(\pi'))$ (resp. $D_6(s)$) has at most a simple pole (resp. a pole of order at most $7$) at $s=1$.  Therefore, $L(s,A^3(\pi)\times(A^2(\pi')\otimes\chi))$ has {\color{red}no exceptional zero} by Proposition~\ref{prop:strategy} applied to $D(s)$ in \eqref{eqn:GL4xGL3}, with $\ell_1 = 8$, $\ell_2 = 0$ and $k \leq 7$.

\subsection{The twist-inequivalent reduced case}
We continue to assume that $\pi\not\sim\pi'$, but now we also assume that $A^3(\pi)$ is not cuspidal.  This introduces two remaining cases.
\subsubsection{$\pi$ is dihedral}
By Lemma~\ref{symmetric power decomposition}(2), $L(s,A^3(\pi) \times (A^2(\pi') \otimes \chi))$ factors as a product of $L$-functions of the form $L(s,\nu \times A^2(\pi') \otimes \chi)$, where $\nu \in \fF_2$ is dihedral. Applying   Proposition~\ref{prop:list}(6) to each factor, we conclude that $L(s,A^3(\pi) \times (A^2(\pi') \otimes \chi))$ has \textcolor{red}{no exceptional zero}.
\subsubsection{$\pi$ is tetrahedral by $\mu$}
By Lemma~\ref{symmetric power decomposition}(3a), $L(s,A^3(\pi) \times A^2(\pi') \otimes \chi)$ decomposes as 
$$L(s,A^3(\pi) \times (A^2(\pi') \otimes \chi)) = L(s,\pi \times (A^2(\pi') \otimes \chi\mu)) \cdot L(s,\pi \times (A^2(\pi') \otimes \chi\mu^2)).$$
Since $\pi \not\sim \pi'$, Theorem~\ref{thm:Langlands}(3) implies that $A^2(\pi) \not\sim A^2(\pi')$. Therefore, $L(s,A^3(\pi) \times A^2(\pi') \otimes \chi)$ has \textcolor{red}{no exceptional zero} by Proposition~\ref{prop:list}(6).

\section{Proof of Theorem~\ref{thm:main}(1c)}
\label{sec:1c}

Let $\pi,\pi'\in\mathfrak{F}_2$.  In this section, we will prove the following result.
\begin{theorem}
\label{thm:GL5xGL2}
Let $\pi,\pi' \in \fF_2$.  
\begin{enumerate}
	\item If $\pi\not\sim\pi'$, then any exceptional zero of $L(s,A^4(\pi)\times\pi')$ is a zero of a self-dual abelian factor.  No such factor exists when $\pi'$ is non-dihedral.
	\item If $\pi\sim\pi'$ and $\pi$ is of solvable polyhedral type, then any exceptional zero of $L(s,A^4(\pi)\times\pi')$ is a zero of a self-dual abelian factor
\end{enumerate}
\end{theorem}
\begin{remark}
\label{rem:2}
If $\pi,\pi'\in\mathfrak{F}_2$, $\psi\in\mathfrak{F}_1$ satisfies $\pi'=\pi\otimes\psi$, and $\pi$ is not of solvable polyhedral type, then
\[
L(s,A^4(\pi)\times\pi')=L(s,\pi,\Sym^5\otimes\bar{\omega}_{\pi}^2\psi)L(s,A^3(\pi)\otimes\psi),\qquad \re(s)>1.
\]
This introduces the difficulties described in Remark \ref{rem:GL4xGL3}.
\end{remark}
\begin{proof}[Proof of Theorem~\ref{thm:main}(1c)]
In Theorem~\ref{thm:GL5xGL2}, replace $\pi'$ with $\pi'\otimes\omega_{\pi}^2$.
\end{proof}

\subsection{The twist-inequivalent general case}
Assume that $\pi \not\sim \pi'$ and $A^4(\pi)$ is cuspidal. 

\subsubsection{$\pi'$ is not self-dual}

Since $A^4(\pi)\in\mathfrak{F}_5$ is self-dual, it follows from Proposition~\ref{prop:list}(8) that if $\pi'$ is not self-dual, then $L(s,A^4(\pi)\times\pi')$ has {\color{red}no exceptional zero}.

\subsubsection{$\pi'$ is self-dual}
\label{subsubsec:5x2_self-dual}

We separately handle the cases where $\pi'$ is dihedral and not.

{\it Case \ref{subsubsec:5x2_self-dual}a: $\pi'$ is dihedral.}  In this case, there exists a nontrivial quadratic $\eta\in\mathfrak{F}_1$ such that $\pi'=\pi'\otimes\eta$.  Note that since $A^4(\pi)$ has trivial character, while $A^4(\pi)\otimes\eta$ has central character $\eta^5=\eta$.  It follows that $A^4(\pi) \neq A^4(\pi) \otimes \eta$. By Theorem~\ref{thm:main3}, we conclude that $L(s,A^4(\pi) \times \pi')$ has \textcolor{red}{no exceptional zero.} 

{\it Case \ref{subsubsec:5x2_self-dual}b: $\pi'$ is non-dihedral.}  Define
\begin{equation}
\label{eqn:GL5xGL2}
\begin{aligned}
    D_4(s) &= \zeta_F(s)^6 \cdot L(s,A^4(\pi) \times A^4(\pi')) \\
    D_5(s) &= L(s,A^2(\pi))^6 \cdot L(s,A^4(\pi))^6 \cdot L(s,\pi')^8 \cdot L(s,A^2(\pi'))^7 \cdot L(s,A^3(\pi'))^4 \cdot L(s,A^4(\pi'))\\
    &\cdot L(s,A^2(\pi)\times\pi')^8 L(s,A^2(\pi)\times A^2(\pi'))^7 L(s,A^2(\pi)\times A^3(\pi'))^4 L(s,A^4(\pi)\times A^2(\pi'))^7\\
    &\cdot L(s,A^2(\pi)\times A^4(\pi')) L(s,A^4(\pi)\times A^3(\pi'))^4,\\
    D(s)&=L(s,A^4(\pi) \times \pi')^8 \cdot D_4(s) \cdot D_5(s).
\end{aligned}
\end{equation}
Since $\pi'$ is assumed to be self-dual, $D(s)$ in \eqref{eqn:GL5xGL2}  matches $D(s)$ in Lemma~\ref{lem:D} with $c_{2,1,0}=2$, $c_{2,0,0}=c_{2,2,0}=1$, and the remaining $c_{j,k,r}$'s equal to zero.  By Lemma~\ref{counting pole lemma}, $D_5(s)$ is entire, while $L(s,A^4(\pi) \times A^4(\pi'))$ (resp. $D_4(s)$) has at most a simple pole (resp. a pole of order at most $7$) at $s=1$.  Therefore, $L(s,A^4(\pi) \times \pi')$ has \textcolor{red}{no exceptional zero} by Proposition~\ref{prop:strategy} applied to $D(s)$ in \eqref{eqn:GL5xGL2}, with $\ell_1=8$, $\ell_2=0$, and $k\leq 7$.

\subsection{The twist-inequivalent reduced case}
We continue to assume that $\pi \not\sim \pi'$, but we now assume that $A^4(\pi)$ is not cuspidal. This introduces three cases.
\subsubsection{$\pi$ is dihedral}
By Lemma~\ref{symmetric power decomposition}(2), $L(s,A^4(\pi) \times \pi')$ factors as a product of $L$-functions of $\GL_1$- or $\GL_2$-twists of $\pi'$. Therefore, by Proposition~\ref{prop:list}(2,4) any exceptional zero of $L(s,A^4(\pi) \times \pi')$ is the \textcolor{blue}{zero of a self-dual abelian $L$-factor}.  Furthermore, by Proposition~\ref{prop:list}(4), if $\pi'$ is non-dihedral, then each factor has \textcolor{red}{no exceptional zero}.
\subsubsection{$\pi$ is tetrahedral by $\mu$}
By Lemma~\ref{symmetric power decomposition}(3a), $L(s,A^4(\pi) \times \pi')$ factors as 
\[
L(s,A^4(\pi) \times \pi') = L(s,A^2(\pi) \times \pi') \cdot L(s,\pi' \otimes \mu) \cdot L(s,\pi' \otimes \mu^2).
\]
By Proposition~\ref{prop:list}(2), $L(s,\pi' \otimes \mu)$ and $L(s,\pi' \otimes \mu^2)$ have no exceptional zero.  By Theorem~\ref{thm:Langlands}(3), our assumption that $\pi \not\sim \pi'$ implies that $A^2(\pi) \not\sim A^2(\pi')$.  Thus, $L(s,A^2(\pi) \times \pi')$ has no exceptional zero by Proposition~\ref{prop:list}(6). We conclude that $L(s,A^4(\pi) \times \pi')$ has \textcolor{red}{no exceptional zero}.
\subsubsection{$\pi$ is octahedral by $\eta$}
By Lemma~\ref{symmetric power decomposition}(3b), there exists a dihedral $\nu\in\mathfrak{F}_2$ such that
\[
L(s,A^4(\pi) \times \pi')  = L(s,A^2(\pi) \times (\pi' \otimes \eta)) \cdot L(s,\nu \times \pi').
\]
By Proposition~\ref{prop:list}(6), $L(s,A^2(\pi) \times (\pi' \otimes \eta))$ has no exceptional zero.  By Proposition~\ref{prop:list}(4), any exceptional zero of $L(s,\nu \times \pi')$ is the zero of a self-dual abelian factor, and no such factor exists when $\pi'$ is non-dihedral.  Therefore, any exceptional zero of $L(s,A^4(\pi) \times \pi')$ is a \textcolor{blue}{zero of a self-dual abelian $L$-factor}, with \textcolor{red}{no exceptional zero} when $\pi'$ is non-dihedral.

\subsection{The twist-equivalent case} 
Suppose that there exists $\mu\in\fF_1$ such that $\pi' = \pi \otimes \mu$, in which case $L(s,A^4(\pi) \times \pi') = L(s,A^4(\pi) \times (\pi \otimes \mu))$.  This introduces three cases.
\subsubsection{$\pi$ is dihedral}
By Lemma~\ref{symmetric power decomposition}(2), $L(s,A^4(\pi) \times (\pi \otimes \mu))$ factors into a product of $\GL_m\times \GL_2$ Rankin--Selberg $L$-functions, where $1 \leq m \leq 2$.  By Proposition~\ref{prop:list}(2,4) applied to each factor, any exceptional zero of $L(s,A^4(\pi) \times \pi')$ is a  \textcolor{blue}{zero of a self-dual abelian factor}.
\subsubsection{$\pi$ is tetrahedral}
By Lemma~\ref{symmetric power decomposition}(3a) and Lemma~\ref{lem:CG}, $L(s,A^4(\pi) \times (\pi \otimes \mu))$ factors as a product of $\GL(2)$ $L$-functions.  Therefore, by Proposition~\ref{prop:list}(2), $L(s,A^4(\pi) \times \pi')$ has \textcolor{red}{no exceptional zero}.
\subsubsection{$\pi$ is octahedral by $\eta$}
By Lemma~\ref{symmetric power decomposition}(3b), there exists a dihedral $\nu\in\mathfrak{F}_2$ such that
\[
L(s,A^4(\pi) \times (\pi \otimes \mu)) = L(s,\pi \otimes \eta\mu) \cdot L(s,\nu \times (\pi \otimes \mu)) \cdot L(s,A^3(\pi) \otimes \eta\mu).
\]
The factor $L(s,\pi\otimes\eta\mu)$ has no exceptional zero by Proposition~\ref{prop:list}(2).  Since $\nu$ is dihedral and $\pi$ is not, $L(s,\nu\times(\pi\otimes\mu))$ has no exceptional zero by Proposition~\ref{prop:list}(4).  Since $\pi$ is octahedral by $\eta$, it follows from Lemma~\ref{symmetric power decomposition}(3b) that $A^3(\pi)\otimes\eta\mu=A^3(\pi)\otimes\mu$.  Therefore, $L(s,A^3(\pi) \otimes \eta\mu)$ has no exceptional zero by Proposition~\ref{prop:list}(3).  In summary, $L(s,A^4(\pi) \times \pi')$ has {\color{red}no exceptional zero}.

\section{Proof of Theorem~\ref{thm:main}(1d)}
\label{sec:1d}

Let $\pi,\pi',\pi'' \in \fF_2$, and assume that $\pi\not\sim\pi''$.  In this section, we will prove the following result.

\begin{theorem}
\label{thm:GL2xGL2xGL2}
If $\pi,\pi',\pi''\in\mathfrak{F}_2$ and $\pi\not\sim\pi''$, then any exceptional zero of $L(s,\pi\times\pi'\times\pi'')$ is a zero of a self-dual abelian factor.  No such factor exists when $\pi,\pi',\pi''$ are all non-dihedral.
\end{theorem}

\begin{remark}
\begin{enumerate}[leftmargin=*]
	\item By the work of Ramakrishnan~\cite{Ramakrishnan}, we can view $L(s,\pi\times\pi'\times\pi'')$ as a $\GL_2\times\GL_4$ Rankin--Selberg $L$-function that might factor.
	\item If $\pi\sim\pi'\sim\pi''$, then there exists $\psi\in\mathfrak{F}_1$ such that $L(s,\pi\times\pi'\times\pi'')=L(s,\pi\otimes\psi)^2 L(s,A^3(\pi)\otimes\psi)$, in which case we run into the difficulties described in Remark \ref{rem:GL4xGL3}.
\end{enumerate}
\end{remark}

\begin{proof}[Proof of Theorem~\ref{thm:main}(1d)]
The necessary conditions are contained in Theorem~\ref{thm:GL2xGL2xGL2}.
\end{proof}

\subsection{The twist-inequivalent non-dihedral case}

Assume that $\pi,\pi',\pi''$ are non-dihedral and pairwise twist-inequivalent.  By Theorem~\ref{thm:Langlands}, there exist representations $\pi\boxtimes\pi',\pi'\boxtimes\pi''\in\mathfrak{F}_4$ and $A^2(\pi)\boxtimes\tilde{\pi}''\in\mathfrak{F}_6$ such that
\[
L(s,\pi\boxtimes\pi')=L(s,\pi\times\pi'),\quad L(s,\pi'\boxtimes\pi'')=L(s,\pi'\times\pi''), \quad L(s,A^2(\pi)\boxtimes\tilde{\pi}'')=L(s,A^2(\pi)\times\tilde{\pi}'').
\]
Define the isobaric sum $\Pi := (\pi \boxtimes \pi') \boxplus \widetilde{\pi}'' \boxplus (A^2(\pi) \boxtimes \widetilde{\pi}'')\in\mathfrak{A}_{12}$ and
\begin{equation}
\label{eqn:GL2xGL2xGL2}
\begin{aligned}
D_8(s) &= \zeta_F(s)^3,  \\
D_9(s) &= L(s,A^2(\pi))^4 \cdot L(s,A^4(\pi))\cdot L(s,A^2(\pi'))\cdot L(s,A^2(\pi''))^2 \\
    & \cdot L(s,A^2(\pi)\times A^2(\pi'))\cdot L(s,A^2(\pi)\times A^2(\pi''))^3\cdot L(s,A^4(\pi)\times A^2(\pi''))\\
    &\cdot L(s,A^3(\pi)\times(\pi'\boxtimes\pi''))\cdot L(s,A^3(\widetilde{\pi})\times(\tilde{\pi}'\boxtimes\tilde{\pi}'')),\\
D(s)&=L(s,\pi \times \pi' \times \pi'')^2 \cdot L(s,\widetilde\pi \times \widetilde{\pi}' \times \widetilde{\pi}'')^2 \cdot D_8(s) \cdot D_9(s).
\end{aligned}
\end{equation}
By Lemma~\ref{lem:AC}, the logarithm of the analytic conductor of $D(s)$ is $O(\log(\mathfrak{C}_{\pi}\mathfrak{C}_{\pi'}\mathfrak{C}_{\pi''}))$. By Lemma~\ref{lem:CG} and Theorem~\ref{thm:Langlands}, we have that $D(s)=L(s,\Pi\times\tilde{\Pi})$.  Each cuspidal constituent of $\Pi$ has a different rank.  Therefore, $D(s)$ has a pole of order 3 at $s=1$ (coming from $D_8(s)$), while $L(s,\pi \times \pi' \times \pi'')$, $L(s,\tilde{\pi} \times \tilde{\pi}' \times \tilde{\pi}'')$, and $D_9(s)$ are entire.  We conclude that $L(s,\pi\times\pi'\times\pi'')$ has \textcolor{red}{no exceptional zero} by Proposition~\ref{prop:strategy} applied to $D(s)$ in \eqref{eqn:GL2xGL2xGL2}, with $\ell_1 = \ell_2 = 2$ and $k=3$.

\subsection{The twist-equivalent non-dihedral case}
We continue to assume that $\pi,\pi',\pi''$ are non-dihedral.  Without loss of generality, we may assume that there exists $\chi\in\mathfrak{F}_1$ such that $\pi' = \pi \otimes \chi$, in which case
\[
L(s,\pi \times \pi' \times \pi'') = L(s,A^2(\pi) \times (\pi'' \otimes \omega_\pi\chi)) \cdot L(s,\pi'' \otimes \omega_\pi\chi).
\]
Since $\pi,\pi''$ are non-dihedral and $\pi \not\sim \pi''$, it follows from Theorem~\ref{thm:Langlands}(3) that $A^2(\pi) \not\sim A^2(\pi'')$. Therefore, by applying by Proposition~\ref{prop:list}(6) to the first factor and Proposition~\ref{prop:list}(2) to the second, $L(s,\pi \times \pi' \times \pi'')$ has \textcolor{red}{no exceptional zero}.

\subsection{The dihedral case}
Assume (without loss of generality) that $\pi''$ is dihedral by $(\eta,\xi,K)$, in which case $\pi'' = I_K^F(\xi)$.  Write $\pi_K$ (resp. $\pi_K'$) for the base change of $\pi$ (resp. $\pi'$) to $K$.  These are (possibly non-cuspidal) automorphic representations of $\GL_2(\mathbb{A}_K)$.  It follows that
\[
L(s,\pi \times \pi' \times \pi'') = L(s,\pi_K \times (\pi'_K \otimes \xi)),
\]
which factors as a product of $\GL_m \times \GL_n$ $L$-functions over $K$, with $m,n \leq 2$.  Applying Proposition~\ref{prop:list}(2,4) to each of these factors, we conclude that any exceptional zero of $L(s,\pi \times \pi' \times \pi'')$ is a \textcolor{blue}{zero of a self-dual abelian factor}.

\section{Proof of Theorem~\ref{thm:main}(1e)}
\label{sec:1e}

Let $(\pi,\pi',\pi_0) \in \fF_2\times \fF_2\times \fF_3$.  Assume that $\pi_0 \not\sim A^2(\pi)$ and $\pi_0 \not\sim A^2(\pi')$. In this section, we prove the following result.

\begin{theorem}
\label{thm:GL2xGL2xGL3}
Let $(\pi,\pi',\pi_0)\in\mathfrak{F}_2\times \mathfrak{F}_2\times \mathfrak{F}_3$.  If $\pi_0\not\sim A^2(\pi)$ and $\pi_0\not\sim A^2(\pi')$, then $L(s,\pi\times\pi'\times\pi_0)$ has no exceptional zero.
\end{theorem}
\begin{remark}
By the work of Kim and Shahidi~\cite{KimShahidi}, we can view $L(s,\pi\times\pi'\times\pi_0)$ as a $\GL_2\times\GL_6$ Rankin--Selberg $L$-function.  If there exists $\psi\in\mathfrak{F}_1$ such that $\pi_0=A^2(\pi)\otimes\psi$, then
\[
L(s,\pi\times\pi'\times\pi_0)=L(s,A^3(\pi)\times(\pi'\otimes\psi))\cdot L(s,\pi\times(\pi'\otimes\psi)),
\]
and Theorem~\ref{thm:GL2xGL4} (below) applies.
\end{remark}
\begin{proof}[Proof of Theorem~\ref{thm:main}(1e)]
The hypotheses are a special case of those in Theorem~\ref{thm:GL2xGL2xGL3}.
\end{proof}

\subsection{The twist-inequivalent non-dihedral case}
Assume that $\pi,\pi'$ are non-dihedral and $\pi \not\sim \pi'$.  Since $\pi_0\not\sim A^2(\pi)$ and $\pi_0\not\sim A^2(\pi')$ by hypothesis, Theorem~\ref{thm:Langlands} ensures that the representations $\pi\boxtimes\pi_0$, $\pi'\boxtimes\pi_0$, and $A^2(\pi)\boxtimes \tilde{\pi}'$ all lie in $\mathfrak{F}_6$.  We may therefore define the isobaric sum
\[
\Pi := (\pi \boxtimes \pi_0) \boxplus \widetilde{\pi}' \boxplus (A^2(\pi) \boxtimes \widetilde{\pi}')\in \fA_{14}
\]
as well as
\begin{equation}
\label{eqn:GL2xGL2xGL3}
\begin{aligned}
    D_{10}(s) &= \zeta_F(s)^2\cdot L(s,(\pi \boxtimes \pi_0) \times (\widetilde\pi \boxtimes \widetilde{\pi}_0)),\\
    D_{11}(s) &= L(s,A^2(\pi))^3\cdot L(s,A^4(\pi)) \cdot L(s,A^2(\pi'))^2 \cdot L(s,A^2(\pi)\times A^2(\pi'))^3 \\
    &\cdot L(s,A^4(\pi)\times A^2(\pi'))\cdot L(s,A^3(\pi)\times(\pi'\boxtimes\pi_0))\cdot L(s,A^3(\widetilde{\pi})\times(\tilde{\pi}'\boxtimes\tilde{\pi}_0)),\\
    D(s)&=L(s,\pi\times\pi'\times\pi_0)^2 \cdot L(s,\tilde{\pi}\times\tilde{\pi}'\times\tilde{\pi}_0)^2 \cdot D_{10}(s)\cdot D_{11}(s).
\end{aligned}
\end{equation}
By Lemma~\ref{lem:AC}, the logarithm of the analytic conductor of $D(s)$ is $O(\log(\mathfrak{C}_{\pi}\mathfrak{C}_{\pi'}\mathfrak{C}_{\pi_0}))$. By Lemma~\ref{lem:CG} and Theorem~\ref{thm:Langlands}, $D(s) = L(s,\Pi\times\tilde{\Pi})$.  Since $\pi \boxtimes \pi_0, A^2(\pi) \boxtimes \tilde{\pi}' \in \fF_6$, it follows that $D_{10}(s)$ has a pole of order $3$ at $s=1$.  Since $\pi'\boxtimes\pi_0\in\mathfrak{F}_6$, the $L$-function $L(s,A^3(\pi) \times (\pi' \boxtimes \pi_0))$ (hence $D_{11}(s)$) is entire.  Therefore, $L(s,\pi \times \pi' \times \pi_0)$ has \textcolor{red}{no exceptional zero} by Proposition~\ref{prop:strategy} applied to $D(s)$ in \eqref{eqn:GL2xGL2xGL3}, with $\ell_1 = \ell_2 = 2$ and $k = 3$. 

\subsection{The twist-equivalent non-dihedral case}
We continue to assume that $\pi,\pi'$ are non-dihedral and $\pi_0 \not\sim A^2(\pi)$, but we now assume that there exists $\chi\in\mathfrak{F}_1$ such that $\pi' = \pi \otimes \chi$, in which case
\[
L(s,\pi \times \pi' \times \pi_0) = L(s,A^2(\pi) \times (\pi_0 \otimes \omega_\pi\chi)) \cdot L(s,\pi_0 \otimes \omega_\pi\chi).
\]
Applying Theorem~\ref{thm:GL3xGL3} to the first factor and Proposition~\ref{prop:list}(2) to the second factor, we conclude that $L(s,\pi \times \pi' \times \pi_0)$ has \textcolor{red}{no exceptional zero}.

\subsection{The dihedral case}
Assume that $\pi$ is dihedral by $(\eta,\xi,K)$, in which case $\pi = I_K^F(\xi)$.  If $\pi'_K$ (resp. $(\pi_0)_K$) is the base change of $\pi$ (resp. $\pi_0$) from $F$ to $K$, then
\[
L(s,\pi \times \pi' \times \pi_0) = L(s,\pi_K' \times ((\pi_0)_K \otimes \xi)).
\]
Let $\theta = \theta_{K/F}$ be the non-trivial element in Gal$(K/F)$, and set $\xi^\theta = \xi \circ \theta$. Since $\eta$ is a quadratic character, we conclude that $\pi_0 \neq \pi_0 \otimes \eta$.  Therefore, by~\cite[Proposition 2.3.1(5)]{Ramakrishnan} (which summarizes the results in~\cite{AC}), $(\pi_0)_K$ is cuspidal.  It remains to consider the following three cases.
\subsubsection{$\pi_K'$ is non-cuspidal}
Here, the base change $\pi'_K$ is a non-cuspidal isobaric automorphic representation of $\GL_2(\mathbb{A}_K)$.  It is therefore an isobaric sum of two idele class characters over $K$.  We already know that $(\pi_0)_K$ is a cuspidal automorphic representation of $\GL_3(\mathbb{A}_K)$, so $L(s,\pi_K' \times (\pi_0)_K \otimes \xi)$ decomposes as a product of $L$-functions of cuspidal automorphic representations of $\GL_3(\mathbb{A}_K)$, \textcolor{red}{which has no exceptional zero} by Proposition~\ref{prop:list}(2).
\subsubsection{$\pi_K'$ is cuspidal, and $\pi'_K \otimes \xi \neq \pi'_K \otimes \xi^\theta$}
By~\cite[Theorem M]{Ramakrishnan}, $\pi \boxtimes \pi'$ is cuspidal. Note that 
\[
(\pi \boxtimes \pi') \otimes \eta = (\pi \otimes \eta) \boxtimes \pi' = \pi \boxtimes \pi'\qquad\textup{and}\qquad\pi_0 \neq \pi_0 \otimes \eta.
\]
Therefore, $L(s,\pi\times \pi'\times \pi_0)$ has \textcolor{red}{no exceptional zero} by Theorem~\ref{thm:main3}.
\subsubsection{$\pi'_K$ is cuspidal, and $\pi'_K \otimes \xi = \pi'_K \otimes \xi^\theta$}
In this case, we have the identity $\pi'_K = \pi'_K \otimes \xi^\theta\xi^{-1}$.  If $M$ is the quadratic extension of $K$ associated to the quadratic character $\xi^\theta\xi^{-1}$ by class field theory, then there exists an idele class character $\psi$ defined over $M$ such that $\pi'_K = I_M^K(\psi)$.  It follows that
\[
L(s,\pi_K' \times ((\pi_0)_K \otimes \xi)) = L(s,((\pi_0)_K \otimes \xi)_M \otimes \psi).
\]
By~\cite[Proposition 2.3.1(5)]{Ramakrishnan}, this is the $L$-function of a cuspidal automorphic representation of $\GL_3(\mathbb{A}_M)$.  Therefore, $L(s,\pi\times \pi'\times \pi_0)$ has \textcolor{red}{no exceptional zero} by Proposition~\ref{prop:list}(2).

\section{Proof of Theorem~\ref{thm:main}(1f)}
\label{sec:1f}

Let $\pi,\pi',\pi''\in\mathfrak{F}_2$, and suppose that $\pi\not\sim\pi'$.  In this section, we will prove the following result.

\begin{theorem}
\label{thm:GL2xGL3xGL3}
If $\pi,\pi',\pi''\in\mathfrak{F}_2$ and $\pi\not\sim\pi'$, then any exceptional zero of $L(s,\pi\times A^2(\pi')\times A^2(\pi''))$ is a zero of a self-dual abelian factor.  No such factor exists when $\pi$ and $\pi'$ are non-dihedral.
\end{theorem}
\begin{remark}
If there exists $\chi\in\mathfrak{F}_1$ such that $\pi'=\pi\otimes\chi$, then
\[
L(s,\pi\times A^2(\pi')\times A^2(\pi'')) = L(s,\pi\times A^2(\pi'')) L(s,A^3(\pi)\times A^2(\pi'')).
\]
The first factor is the subject of Lemma~\ref{lem:Luo}.  The second factor is the subject of Theorem~\ref{thm:GL4xGL3} and the remark that follows it.
\end{remark}

\begin{proof}[Proof of Theorem~\ref{thm:main}(1f)]
In Theorem~\ref{thm:GL2xGL3xGL3}, replace $\pi$ with $\pi\otimes\omega_{\pi'}\omega_{\pi''}$.
\end{proof}

\subsection{The twist-inequivalent general case}
Assume that $\pi,\pi',\pi''$ are pairwise twist-inequivalent, and assume that $\pi'$ and $\pi''$ are non-dihedral. We discuss the following three cases.
\subsubsection{$\pi$ is non-dihedral, and $\pi'$ or $\pi''$ is non-tetrahedral}
It follows that $\pi,\pi',\pi''$ are non-dihedral and pairwise twist-inequivalent.  We now generalize the approach described in Section \ref{subsec:OurApproach}, crucially using Theorem~\ref{thm:Langlands} to ensure that $\pi \boxtimes A^2(\pi'),\pi \boxtimes A^2(\pi'')\in\mathfrak{F}_6$.  Define
\begin{equation}
\label{eqn:GL2xGL3xGL3}
\begin{aligned}
    D_{12}(s) &= \zeta_F(s)^6 \cdot L(s,A^4(\pi') \times A^4(\pi'')), \\
    D_{13}(s) &=L(s,A^2(\pi))^4 \cdot L(s,A^2(\pi'))^7 \cdot L(s,A^4(\pi'))^5 \cdot L(s,A^2(\pi''))^2\cdot L(s,A^4(\pi''))^2\\
    &\cdot L(s,\pi \times A^2(\pi''))^2 \cdot L(s,\widetilde\pi \times A^2(\pi''))^2 \cdot L(s,A^2(\pi')\times A^2(\pi''))^3  \\
    &\cdot L(s,A^2(\pi')\times A^4(\pi''))^3\cdot L(s,A^2(\pi)\times A^2(\pi'))^4\cdot L(s,A^2(\pi)\times A^4(\pi'))^4\\
    &\cdot L(s,A^4(\pi')\times A^2(\pi''))\cdot L(s,(\pi\boxtimes A^2(\pi''))\times A^4(\pi'))^2 \cdot L(s,(\tilde{\pi}\boxtimes A^2(\pi''))\times A^4(\pi'))^2,\\
    D(s)&=L(s,(\pi \boxtimes A^2(\pi')) \times A^2(\pi''))^4 \cdot L(s,(\widetilde\pi \boxtimes A^2(\pi')) \times A^2(\pi''))^4 \cdot D_{12}(s) \cdot D_{13}(s).
\end{aligned}\hspace{-2mm}
\end{equation}
By Lemma~\ref{lem:AC}, the logarithm analytic conductor of $D(s)$ is $O(\log(\mathfrak{C}_{\pi}\mathfrak{C}_{\pi'}\mathfrak{C}_{\pi''}))$.
\begin{lemma}
\label{lem:GL2xGL3xGL3}
Let $D(s)$ be as in \eqref{eqn:GL2xGL3xGL3}, and let $a_D(v^{\ell})\log q_v$ be the $v^{\ell}$-th Dirichlet coefficient of $-(D'/D)(s)$.  If $v\notin S_{\pi}^{\infty}\cup S_{\pi'}^{\infty}\cup S_{\pi''}^{\infty}$, then $a_D(v^{\ell})\geq 0$.
\end{lemma}
\begin{proof}
Let $v\notin S_{\pi}^{\infty}\cup S_{\pi'}^{\infty}\cup S_{\pi''}^{\infty}$, and define $\Pi_v = 2\pi_v\otimes A^2(\pi_v')\oplus A^2(\pi_v'')\oplus A^2(\pi_v')\otimes A^2(\pi_v'')$.  On one hand, it follows from Lemma~\ref{lem:CG} and Theorem~\ref{thm:Langlands} that the $v$-th Euler factor of $D(s)$ is $L(s,\Pi_v\otimes\tilde{\Pi}_v)$.  On the other hand, one computes
\[
-\frac{L'}{L}(s,\Pi_v\otimes\tilde{\Pi}_v)=\sum_{\ell=1}^{\infty}\frac{|2a_{\pi}(v^{\ell})a_{A^2(\pi')}(v^{\ell})+a_{A^2(\pi'')}(v^{\ell})+a_{A^2(\pi')}(v^{\ell})a_{A^2(\pi'')}(v^{\ell})|^2\log q_v}{q_v^{\ell s}}.
\]
The desired result follows.
\end{proof}

Since $A^4(\pi') \in \fA_5$ and $\pi \boxtimes A^2(\pi'')\in\mathfrak{F}_6$, $L(s,(\pi \boxtimes A^2(\pi'')) \times A^4(\pi'))$ is entire.  By Lemma~\ref{counting pole lemma}, all other factors of $D_{13}(s)$ are entire, hence $D_{13}(s)$ is entire.  Moreover, by Lemma~\ref{counting pole lemma}(3,4) and our assumption that $\pi',\pi''$ are not both tetrahedral, $L(s,A^4(\pi') \times A^4(\pi''))$ (resp. $D_{12}(s)$) has at most a simple pole (resp. has a pole of order at most $7$) at $s=1$. Therefore, $L(s,(\pi \boxtimes A^2(\pi')) \times A^2(\pi''))$ has \textcolor{red}{no exceptional zero} by Proposition~\ref{prop:strategy} applied to $D(s)$ in \eqref{eqn:GL2xGL3xGL3}, with $\ell_1 = \ell_2 = 4$ and $k \leq 7$.
\subsubsection{$\pi$ is non-dihedral, $\pi'$ and $\pi''$ are tetrahedral}
If $\pi'$ is tetrahedral by $\mu$, then Lemma~\ref{symmetric power decomposition}(3a) implies that
\[
L(s,(\pi \boxtimes A^2(\pi'')) \times A^4(\pi')) = L(s,(\pi \boxtimes A^2(\pi')) \times A^2(\pi'')) \cdot L(s,\pi \times (A^2(\pi'') \otimes \mu)) \cdot L(s,\pi \times (A^2(\pi'') \otimes \mu^2)).
\]
Using this decomposition, we find that if
\begin{align*}
D_{14}(s) &=L(s,A^2(\pi))^4 \cdot L(s,A^2(\pi'))^7 \cdot L(s,A^4(\pi'))^5 \cdot L(s,A^2(\pi''))^2\cdot L(s,A^4(\pi''))^2\\
    &\cdot L(s,\pi \times A^2(\pi''))^2 \cdot L(s,\widetilde\pi \times A^2(\pi''))^2 \cdot L(s,A^2(\pi')\times A^2(\pi''))^3  \\
    &\cdot L(s,A^2(\pi')\times A^4(\pi''))^3\cdot L(s,A^2(\pi)\times A^2(\pi'))^4\cdot L(s,A^2(\pi)\times A^4(\pi'))^4\\
    &\cdot L(s,A^4(\pi')\times A^2(\pi''))\cdot L(s,\pi \times (A^2(\pi'') \otimes \mu))^2 \cdot L(s,\pi \times (A^2(\pi'') \otimes \mu^2))^2\\
    &\cdot L(s,\tilde{\pi} \times (A^2(\pi'') \otimes \bar{\mu}))^2 \cdot L(s,\tilde{\pi} \times (A^2(\pi'') \otimes \bar{\mu}^2))^2,
\end{align*}
then $D(s)$ in \eqref{eqn:GL2xGL3xGL3} satisfies
\begin{equation}
\label{eqn:GL2xGL3xGL3_2}
D(s)= L(s,\pi \times A^2(\pi') \times A^2(\pi''))^6 \cdot L(s,\widetilde\pi \times A^2(\pi') \times A^2(\pi''))^6 \cdot D_{12}(s) \cdot D_{14}(s).
\end{equation}

By Lemma~\ref{counting pole lemma}, $D_{14}(s)$ is entire, and $L(s,A^4(\pi') \times A^4(\pi''))$ (resp. $D_{12}(s)$) has a pole of order at most 4 (resp. at most $10$) at $s=1$.  Therefore, $L(s,\pi \times A^2(\pi') \times A^2(\pi''))$ has \textcolor{red}{no exceptional zero} by Proposition~\ref{prop:strategy} applied to $D(s)$ in \eqref{eqn:GL2xGL3xGL3_2}, with $\ell_1 = \ell_2 = 6$ and $k \leq 10$.

\subsubsection{$\pi$ is dihedral}
\label{subsubsec:GL2xGL3xGL3_pi_dihedral}

We now assume that $\pi$ is dihedral by $(\eta,\xi,K)$.  This introduces two cases.

{\it Case \ref{subsubsec:GL2xGL3xGL3_pi_dihedral}a:  $\pi \boxtimes A^2(\pi')$ is cuspidal.}  In this case, we have that
\[
(\pi \boxtimes A^2(\pi')) \otimes \eta = (\pi \otimes \eta) \boxtimes A^2(\pi') = \pi \boxtimes A^2(\pi').
\]
Since $\eta$ is a non-trivial quadratic character, we also have that $A^2(\pi'') \neq A^2(\pi'') \otimes \eta$.  Therefore, by Theorem~\ref{thm:main3}, $L(s,\pi \times A^2(\pi') \times A^2(\pi''))$ has \textcolor{red}{no exceptional zero.}

{\it Case \ref{subsubsec:GL2xGL3xGL3_pi_dihedral}b:  $\pi \boxtimes A^2(\pi')$ is not cuspidal.}  By~\cite[Theorem 9.1]{RamakrishnanWang2}, there exist $\chi_1,\chi_2\in\mathfrak{F}_1$ such that $\pi\boxtimes A^2(\pi')=A^2(\pi')\otimes\chi_1\boxplus A^2(\pi')\otimes\chi_2$.  It follows that
\[
L(s,(\pi \boxtimes A^2(\pi')) \times A^2(\pi'')) = L(s,A^2(\pi') \times (A^2(\pi'') \otimes \chi_1)) \cdot L(s,A^2(\pi') \times (A^2(\pi'') \otimes \chi_2)).
\]
Applying Theorem~\ref{thm:GL3xGL3} to each factor, we conclude that $L(s,(\pi \boxtimes A^2(\pi')) \times A^2(\pi''))$ has \textcolor{red}{no exceptional zero}.

\subsection{The twist-inequivalent reduced case}
Assume that $\pi,\pi',\pi''$ are pairwise twist-inequivalent, and at least one of $\pi'$ and $\pi''$ is dihedral.  Without loss of generality, we assume that $\pi'$ is dihedral by $(\eta,\xi,K)$.  This introduces the following two cases.
\subsubsection{$\pi''$ is non-dihedral}
By Lemma~\ref{symmetric power decomposition}(2), we have the identity
\[
L(s,(\pi \boxtimes A^2(\pi')) \times A^2(\pi'')) = L(s,(\pi  \boxtimes A^2(\pi''))\times I_K^F(\xi{\xi'}^{-1})) \cdot L(s,\pi \times (A^2(\pi'') \otimes \eta)),
\]
Since $\pi''$ is non-dihedral, $A^2(\pi'')\in\mathfrak{F}_3$ while $A^2(I_K^F(\xi{\xi'}^{-1}))$ is not, hence $A^2(\pi'')\not\sim A^2(I_K^F(\xi{\xi'}^{-1}))$.  Since $\pi\not\sim\pi''$ by hypothesis, it follows from Theorem~\ref{thm:Langlands}(3) that $A^2(\pi)\not\sim A^2(\pi'')$.  By applying  Theorem~\ref{thm:GL2xGL2xGL3} to the first factor and Proposition~\ref{prop:list}(6) to the second factor, we conclude that $L(s,(\pi \boxtimes A^2(\pi')) \times A^2(\pi''))$ has \textcolor{red}{no exceptional zero}.
\subsubsection{$\pi''$ is dihedral}
It follows from Lemma~\ref{symmetric power decomposition}(2), that there exist non-trivial quadratic characters $\eta',\eta'' \in \fF_1$ and dihedral $\nu',\nu'' \in \fF_2$ such that
\[
    L(s,(\pi \boxtimes A^2(\pi')) \times A^2(\pi'')) = L(s,(\pi \boxtimes \nu') \times \nu'') \cdot L(s,\pi \times (\nu' \otimes \eta'')) \cdot L(s,\pi \times (\nu'' \otimes \eta')) \cdot L(s,\pi \otimes \eta'\eta'').
\]
Applying Theorem~\ref{thm:GL2xGL2xGL2} to the first factor and Proposition~\ref{prop:list}(2,4) to the others, we conclude that any exceptional zero of $L(s,(\pi\boxtimes A^2(\pi'))\times A^2(\pi''))$ is a \textcolor{blue}{zero of a self-dual abelian $L$-factor}.

\subsection{The twist-equivalent case}
Assume that exactly two of $\pi,\pi',\pi''$ are twist-equivalent.  Without loss of generality, this introduces two cases.
\subsubsection{$\pi\not\sim\pi'$ and $\pi' \sim \pi''$}
In this case we have $A^2(\pi') = A^2(\pi'')$, and Lemma~\ref{lem:CG} implies that
\[
L(s,\pi \times A^2(\pi') \times A^2(\pi'')) = L(s,\pi) \cdot L(s,\pi \times A^2(\pi')) \cdot L(s,\pi \times A^4(\pi')).
\]
Applying Proposition~\ref{prop:list}(2) to the first factor, Proposition~\ref{prop:list}(6) to the second, and Theorem~\ref{thm:GL5xGL2} to the third, we conclude that any exceptional zero of $L(s,\pi \times A^2(\pi') \times A^2(\pi''))$ is a \textcolor{blue}{zero of a self-dual abelian factor}.  Also, if $\pi$ and $\pi'$ are non-dihedral, then there is \textcolor{red}{no exceptional zero}.
\subsubsection{$\pi\not\sim\pi'$ and $\pi\sim\pi''$}
Here, we have that $A^2(\pi'') = A^2(\pi)$, so Lemma~\ref{lem:CG} implies that
\[
L(s,\pi \times A^2(\pi') \times A^2(\pi'')) = L(s,\pi \times A^2(\pi')) \cdot L(s,A^3(\pi) \times A^2(\pi')).
\]
Applying Proposition~\ref{prop:list}(6) to the first factor and Theorem~\ref{thm:GL4xGL3} to the second factor, we conclude that any exceptional zero of $L(s,\pi \times A^2(\pi') \times A^2(\pi''))$ is a  \textcolor{blue}{zero of a self-dual abelian factor}.  In particular, if $\pi$ and $\pi'$ are non-dihedral, then there is \textcolor{red}{no exceptional zero}.

\section{Proof of Theorem~\ref{thm:main}(2)}
\label{sec:2}

Let $\pi,\pi'\in\mathfrak{F}_2$.  In this section, we prove the following result.

\begin{theorem}
\label{thm:GL2xGL4}
Let $\pi,\pi'\in\mathfrak{F}_2$.  If $A^3(\pi)\neq A^3(\widetilde{\pi}')$ or $\pi\sim\pi'$, then any exceptional zero of $L(s,A^3(\pi)\times\pi')$ is a zero of a self-dual abelian factor.  If $\pi\not\sim\pi'$ and at least one of $\pi,\pi'$ is non-dihedral, then no such factor exists.
\end{theorem}

\begin{proof}[Proof of Theorem~\ref{thm:main}(2)]
In Theorem~\ref{thm:GL2xGL4}, replace $\pi'$ with $\pi'\otimes\omega_{\pi}$.
\end{proof}

\subsection{The twist-inequivalent general case} Assume that $\pi \not\sim \pi'$ and $A^3(\pi)\in\mathfrak{F}_4$.  Define
\begin{equation}
\label{eqn:GL4xGL2}
\begin{aligned}
    D_{15}(s) &= \zeta_F(s)^9 \cdot L(s,A^3(\pi) \times A^3(\pi'))^4 \cdot L(s,A^3(\widetilde{\pi})\times A^3(\widetilde{\pi}'))^4 \cdot L(s,A^4(\pi'))  \\
    D_{16}(s) &= L(s,\pi \times \pi')^6 \cdot L(s,\widetilde{\pi} \times \widetilde{\pi}')^6 \cdot L(s,\pi \times A^3(\pi'))^4 \cdot L(s,\widetilde\pi \times A^3(\widetilde{\pi}'))^4 \\
    &\cdot L(s,A^2(\pi'))^9 \cdot L(s,A^2(\pi) \times A^2(\pi'))^4 \cdot L(s,A^4(\pi))^4 \cdot L(s,A^4(\pi) \times A^2(\pi'))^4 \\
    &\cdot L(s,A^2(\pi))^9 \cdot L(s,A^2(\pi) \times A^2(\pi'))^5 \cdot L(s,A^2(\pi) \times A^4(\pi')),\\
    D(s)&= L(s,A^3(\pi) \times \pi')^6 \cdot L(s,A^3(\widetilde{\pi})\times \widetilde{\pi}')^6 \cdot D_{15}(s) \cdot D_{16}(s).
\end{aligned}
\end{equation}
The $D(s)$ in \eqref{eqn:GL4xGL2} matches $D(s)$ in Lemma~\ref{lem:D} with $c_{1,2,0} = 1$, $c_{1,0,0} =c_{2,1,0} = 2$, and $c_{j,k,r}=0$ otherwise.  If $\pi'$ is non-dihedral, then we apply Lemma~\ref{counting pole lemma} to see that $D_{16}(s)$ is entire, and $L(s,A^4(\pi'))$ is entire. If $\pi'$ is dihedral, then we use the assumption $\pi \not\sim \pi'$ and Lemma~\ref{symmetric power decomposition}(2) to see that $D_{16}(s)$ is again entire, and $L(s,A^4(\pi'))$ has at most a double pole at $s=1$. In either case, if $A^3(\pi)\neq A^3(\widetilde{\pi}')$ (i.e., $L(s,A^3(\pi) \times A^3(\pi'))$ is entire), then $L(s,A^3(\pi) \times \pi')$ has \textcolor{red}{no exceptional zero} by Proposition~\ref{prop:strategy} applied to $D(s)$ in \eqref{eqn:GL4xGL2}, with $\ell_1 = \ell_2 = 6$ and $k \leq 11$. 

\subsection{The twist-inequivalent reduced case}
We continue to assume that $\pi \not\sim \pi'$, but we now assume that $A^3(\pi)$ is not cuspidal. We have two cases to consider.
\subsubsection{$\pi$ is dihedral}
By Lemma~\ref{symmetric power decomposition}(2), there exist dihedral representations $\nu_1,\nu_2\in\mathfrak{F}_2$ such that $L(s,A^3(\pi) \times \pi')$ factors as $L(s,\nu_1\times\pi')\cdot L(s,\nu_2\times\pi')$.  Applying Proposition~\ref{prop:list}(4) to each factor, we conclude that any exceptional zero of $L(s,A^3(\pi) \times \pi')$ is a \textcolor{blue}{zero of a self-dual abelian factor}, with  \textcolor{red}{no exceptional zero} when $\pi'$ is non-dihedral.
\subsubsection{$\pi$ is tetrahedral by $\mu$}
By Lemma~\ref{symmetric power decomposition}(3a), we have that
\[
L(s,A^3(\pi) \times \pi') = L(s,\pi \times (\pi' \otimes \mu)) \cdot L(s,\pi \times (\pi' \otimes \mu^2)).
\]
Since $\pi$ is non-dihedral and $\pi\not\sim\pi'$, it follows from Proposition~\ref{prop:list}(4) that $L(s,A^3(\pi) \times \pi')$ has \textcolor{red}{no exceptional zero}.

\subsection{The twist-equivalent case}
If $\psi\in\mathfrak{F}_1$ satisfies $\pi' = \pi \otimes \psi$, then
\[
    L(s,A^3(\pi) \times \pi') = L(s,A^3(\pi) \times (\pi \otimes \eta)) = L(s,A^4(\pi) \otimes \omega_\pi\psi) \cdot L(s,A^2(\pi) \otimes \omega_\pi\psi)
\]
by Lemma~\ref{lem:CG}.  By Theorem~\ref{symmetric fourth no Siegel zero} applied to the first factor and Proposition~\ref{prop:list}(2) applied to the second factor, any exceptional zero of $L(s,A^3(\pi) \times \pi')$ is a \textcolor{blue}{zero of a self-dual abelian factor}.

\section{Proof of Theorem~\ref{thm:main}(3)}
\label{sec:3}

Let $\chi\in\mathfrak{F}_1$, and let $\pi,\pi'\in\mathfrak{F}_2$ satisfy $\pi\not\sim\pi'$.  In this section, we prove the following result.

\begin{theorem}
\label{thm:GL5xGL3}
Let $\pi,\pi'\in\mathfrak{F}_2$ and $\chi\in\mathfrak{F}_1$.  Suppose that $\pi\not\sim\pi'$.  If $\chi^2\neq\mathbbm{1}$ or $A^4(\pi)\neq A^4(\pi')$, then any exceptional zero of $L(s,A^4(\pi)\times(A^2(\pi')\otimes\chi))$ is a zero of a self-dual abelian factor.  No such factor exists when $\pi'$ is non-dihedral.
\end{theorem}
\begin{remark}
If there exists $\psi\in\mathfrak{F}_1$ such that $\pi'=\pi\otimes\psi$, then $L(s,A^4(\pi)\times(A^2(\pi')\otimes\chi))=L(s,A^4(\pi)\times(A^2(\pi)\otimes\chi))$.  When $\pi$ is of solvable polyhedral type, this factors according to Lemma~\ref{symmetric power decomposition}, and any exceptional zero is a zero of a self-dual abelian factor.  If $\pi$ is not of solvable polyhedral type, then we have the factorization
\begin{align*}
L(s,A^4(\pi)\times(A^2(\pi')\otimes\chi))
&=L(s,A^2(\pi)\otimes\chi)\cdot L(s,A^4(\pi)\otimes\chi)\cdot L(s,\pi,\Sym^6\otimes\chi\bar{\omega}_{\pi}^3),
\end{align*}
which is related to Corollary~\ref{cor:sym2}.  Outside of the context of Proposition~\ref{prop:Thorner}, we cannot yet preclude the existence of exceptional zeros for $L(s,\pi,\Sym^6\otimes\chi\bar{\omega}_{\pi}^3)$.
\end{remark}
\begin{proof}[Proof of Theorem~\ref{thm:main}(3)]
In Theorem~\ref{thm:GL5xGL3}, replace $\chi$ with $\chi\omega_{\pi}^2\omega_{\pi'}$.
\end{proof}

\subsection{The twist-inequivalent general case}

Suppose that $\pi \not\sim \pi'$, $A^2(\pi')\in\mathfrak{F}_3$, and  $A^4(\pi)\in\mathfrak{F}_5$.  Since $A^4(\pi)$ is self-dual, we consider two cases.

\subsubsection{$A^2(\pi') \otimes \chi$ is not self-dual}

In this case, $L(s,A^4(\pi) \times (A^2(\pi') \otimes \chi))$ has \textcolor{red}{no exceptional zero} by Proposition~\ref{prop:list}(8). 

\subsubsection{$A^2(\pi') \otimes \chi$ is self-dual}
\label{A^2(pi')otimeschiisself-dual}

The self-duality of $A^2(\pi')$ and $A^2(\pi') \otimes \chi$ imply that
\begin{equation}
\label{eqn:A2_twist_SD}
A^2(\pi') \otimes \chi^2 = A^2(\pi').
\end{equation}
By comparing central characters, we conclude that $\chi^6=\mathbbm{1}$.

{\it Case \ref{A^2(pi')otimeschiisself-dual}a: $\chi=\mathbbm{1}$}.  Define
\begin{equation}
\label{eqn:GL5xGL3}
\begin{aligned}
    D_{17}(s) &= \zeta_F(s)^6 \cdot L(s,A^4(\pi) \times A^4(\pi')) \\
    D_{18}(s) &= L(s,A^2(\pi))^6\cdot L(s,A^4(\pi))^6\cdot L(s,\pi')^4 \cdot L(s,\widetilde{\pi}')^4 \cdot L(s,A^2(\pi'))^7\cdot L(s,A^3(\pi'))^2\\
    & \cdot L(s,A^3(\widetilde{\pi}'))^2 \cdot L(s,A^4(\pi')) \cdot L(s,A^2(\pi) \times \pi')^4 \cdot L(s,A^2(\pi) \times \widetilde{\pi}')^4\\
    &\cdot L(s,A^4(\pi) \times \pi')^4 \cdot L(s,A^4(\pi) \times \widetilde{\pi}')^4 \cdot L(s,A^2(\pi) \times A^2(\pi'))^7\cdot L(s,A^2(\pi) \times A^4(\pi'))\\
    &\cdot L(s,A^4(\pi) \times A^3(\pi'))^2 \cdot L(s,A^4(\pi) \times A^3(\widetilde{\pi}'))^2 \\ 
    &\cdot L(s,A^2(\pi) \times A^3(\pi'))^2 \cdot L(s,A^2(\pi) \times A^3(\widetilde{\pi}'))^2,  \\
    D(s) &= L(s,A^4(\pi) \times A^2(\pi'))^7 \cdot D_{17}(s) \cdot D_{18}(s).
\end{aligned}
\end{equation}
The $D(s)$ in \eqref{eqn:GL5xGL3} matches $D(s)$ in Lemma~\ref{lem:D} with $c_{2,1,0} = 2$, $c_{2,0,0}=c_{2,2,0} = 1$, and $c_{j,k,r}=0$ otherwise.  By Lemma~\ref{counting pole lemma}, $D_{18}(s)$ is entire. If $A^4(\pi)\neq A^4(\pi')$ (i.e., $L(s,A^4(\pi) \times A^4(\pi'))$ is entire), then $D_{17}(s)$ has a pole of order $6$ at $s=1$.  It follows that $L(s,A^4(\pi) \times A^2(\pi'))$ \textcolor{red}{has no exceptional zero} by Proposition~\ref{prop:strategy} applied to $D(s)$ in \eqref{eqn:GL5xGL3}, with $\ell_1 = 7$, $\ell_2 = 0$, and $k = 6$.

{\it Case \ref{A^2(pi')otimeschiisself-dual}b: $\chi\neq \mathbbm{1}$ and $\chi^2=\mathbbm{1}$}.  Define
\begin{equation}
\label{eqn:GL5xGL3_2}
\begin{aligned}
    D_{19}(s) &= \zeta_F(s)^6 \cdot L(s,A^4(\pi)\times A^4(\pi'))^5\cdot L(s,A^4(\pi)\times(A^4(\pi')\otimes\chi))^4,\\
    D_{20}(s) &= L(s,\chi)^4\cdot L(s,A^2(\pi))^6\cdot L(s,A^2(\pi)\otimes\chi)^4\cdot L(s,A^4(\pi))^6\cdot L(s,A^4(\pi)\otimes\chi)^4\\
    &\cdot L(s,A^2(\pi'))^7\cdot L(s,A^2(\pi')\otimes\chi)^8\cdot L(s,A^4(\pi'))^5\cdot L(s,A^4(\pi')\otimes\chi)^4\\
    &\cdot L(s,A^2(\pi)\times A^2(\pi'))^7\cdot L(s,A^2(\pi)\times(A^2(\pi')\otimes\chi))^8\cdot L(s,A^2(\pi)\times A^4(\pi'))^5\\
    &\cdot L(s,A^2(\pi)\times(A^4(\pi')\otimes\chi))^4\cdot L(s,A^4(\pi)\times A^2(\pi'))^7,\\
D(s)&=L(s,A^4(\pi)\times (A^2(\pi')\otimes\chi))^8\cdot D_{19}(s)\cdot D_{20}(s).
\end{aligned}
\end{equation}
The $D(s)$ in \eqref{eqn:GL5xGL3_2} matches $D(s)$ in Lemma~\ref{lem:D} with $c_{2,2,0} = 2$, $c_{2,0,1} =c_{2,2,1} = 1$, and $c_{j,k,r}=0$ otherwise.  By Lemma~\ref{counting pole lemma}, $D_{20}(s)$ is entire.

The central characters of $A^4(\pi)$ and $A^4(\pi')\otimes\chi$ are $\mathbbm{1}$ and $\chi^5=\chi$, respectively.  Since $\chi\neq\mathbbm{1}$ by hypothesis, we conclude that $A^4(\pi)\neq A^4(\pi')\otimes\chi$, which implies that $L(s,A^4(\pi)\times(A^4(\pi')\otimes\chi))$ is entire.  Therefore, if $A^4(\pi)\neq A^4(\pi')$, then $L(s,A^4(\pi)\times A^4(\pi'))$ is entire, and  $D_{19}(s)$ has a pole of order $6$ at $s=1$.  We conclude that $L(s,A^4(\pi) \times (A^2(\pi') \otimes \chi))$ has \textcolor{red}{no exceptional zero} by Proposition~\ref{prop:strategy} applied to $D(s)$ in \eqref{eqn:GL5xGL3_2}, with $\ell_1=8$, $\ell_2=0$, and $k=6$.

{\it Case \ref{A^2(pi')otimeschiisself-dual}c:  $\chi^2 \neq \mathbbm{1}, \chi^6 = \mathbbm{1}$.} 
Twisting both sides of \eqref{eqn:A2_twist_SD} by $\chi$, we conclude that $A^2(\pi') \otimes \chi  = (A^2(\pi') \otimes \chi) \otimes  \chi^2$, where $\chi^2$ is non-trivial.  Since the central characters of $A^4(\pi)$ and $A^4(\pi)\otimes\chi^2$ are $\mathbbm{1}$ and $\chi^{10}$, respectively, our conditions on $\chi$ ensure that $A^4(\pi)\neq A^4(\pi)\otimes\chi^2$.  Therefore, $A^2(\pi')\otimes\chi$ has a self-twist by $\chi^2$, but $A^4(\pi)$ does not, in which case $L(s,A^4(\pi)\times(A^2(\pi')\otimes\chi))$ has \textcolor{red}{no exceptional zero} by Theorem~\ref{thm:main3}.

\subsection{The twist-inequivalent reduced case}
We continue to assume that $\pi \not\sim \pi'$, but we now assume that $A^4(\pi)$ or $A^2(\pi')$ is non-cuspidal. This introduces four cases.
\subsubsection{$\pi'$ is dihedral by $(\eta,\xi,K)$}
By Lemma~\ref{symmetric power decomposition}(2), $L(s,A^4(\pi) \times (A^2(\pi') \otimes \chi))$ decomposes as
\[
L(s,A^4(\pi) \times (A^2(\pi') \otimes \chi)) = L(s,A^4(\pi) \times (I_K^F(\xi{\xi'}^{-1}) \otimes \chi)) \cdot L(s,A^4(\pi) \otimes \eta\chi).
\]
By Theorems~\ref{symmetric fourth no Siegel zero}~and~\ref{thm:GL5xGL2}, any exceptional zero of the factors on the right-hand side is a \textcolor{blue}{zero of a self-dual abelian factor}.
\subsubsection{$\pi'$ is non-dihedral, $\pi$ is dihedral}
By Lemma~\ref{symmetric power decomposition}(2), $L(s,A^4(\pi) \times (A^2(\pi') \otimes \chi))$ factors as a product of $L$-functions of the form $L(s,A^2(\pi') \otimes \eta\chi)$ or $L(s,\nu \times (A^2(\pi') \otimes \chi))$, where $\nu \in \fF_2$ is dihedral and $\eta \in \fF_1$.  By Proposition~\ref{prop:list}(2,6), each factor has \textcolor{red}{no exceptional zero}.
\subsubsection{$\pi'$ is non-dihedral, $\pi$ is tetrahedral by $\mu$}
By Lemma~\ref{symmetric power decomposition}(3a), we have the factorization
\[
L(s,A^4(\pi) \times (A^2(\pi') \otimes \chi)) = L(s,A^2(\pi) \times (A^2(\pi') \otimes \chi)) \cdot L(s,A^2(\pi') \otimes \mu\chi) \cdot L(s,A^2(\pi') \otimes \mu^2\chi).
\]
Since $\pi \not\sim \pi'$, Theorem~\ref{thm:Langlands}(3) implies that $A^2(\pi') \not\sim A^2(\pi)$.  Applying Theorem~\ref{thm:GL3xGL3} to the first factor and Proposition~\ref{prop:list}(2) to the others, we conclude that $L(s,A^4(\pi) \times (A^2(\pi') \otimes \chi))$ has 
\textcolor{red}{no exceptional zero}.
\subsubsection{$\pi'$ is non-dihedral, $\pi$ is octahedral by $\eta$}
By Lemma~\ref{symmetric power decomposition}(3b), there exists a dihedral $\nu \in \fF_2$ such that $L(s,A^4(\pi) \times (A^2(\pi') \otimes \chi)) = L(s,A^2(\pi) \times (A^2(\pi') \otimes \eta\chi)) \cdot L(s,\nu \times (A^2(\pi') \otimes \chi))$.  Applying Theorem~\ref{thm:GL3xGL3} to the first factor and Proposition~\ref{prop:list}(6) to the second factor, we conclude that $L(s,A^4(\pi) \times (A^2(\pi') \otimes \chi))$ \textcolor{red}{has no exceptional zero}.

\bibliographystyle{abbrv}
\bibliography{ThornerZhao}

\begin{thebibliography}{10}

\bibitem{AC}
J.~Arthur and L.~Clozel.
\newblock {\em Simple algebras, base change, and the advanced theory of the
  trace formula}, volume 120 of {\em Annals of Mathematics Studies}.
\newblock Princeton University Press, Princeton, NJ, 1989.

\bibitem{Banks}
W.~D. Banks.
\newblock Twisted symmetric-square {$L$}-functions and the nonexistence of
  {S}iegel zeros on {$\mathrm{GL}(3)$}.
\newblock {\em Duke Math. J.}, 87(2):343--353, 1997.

\bibitem{Brumley}
F.~Brumley.
\newblock Effective multiplicity one on $\mathrm{GL}_n$ and narrow zero-free
  regions for {R}ankin-{S}elberg {$L$}-functions.
\newblock {\em Amer. J. Math.}, 128(6):1455--1474, 2006.

\bibitem{BushnellHenniart}
C.~J. Bushnell and G.~Henniart.
\newblock An upper bound on conductors for pairs.
\newblock {\em J. Number Theory}, 65(2):183--196, 1997.

\bibitem{GJ}
S.~Gelbart and H.~Jacquet.
\newblock A relation between automorphic representations of {$\mathrm{GL}(2)$}
  and {$\mathrm{GL}(3)$}.
\newblock {\em Ann. Sci. \'{E}cole Norm. Sup. (4)}, 11(4):471--542, 1978.

\bibitem{GodementJacquet}
R.~Godement and H.~Jacquet.
\newblock {\em Zeta functions of simple algebras}.
\newblock Lecture Notes in Mathematics, Vol. 260. Springer-Verlag, Berlin-New
  York, 1972.

\bibitem{HarcosThorner}
G.~{Harcos} and J.~{Thorner}.
\newblock {A new zero-free region for Rankin-Selberg $L$-functions}.
\newblock {\em J. Reine Angew. Math.}
\newblock Accepted for publication.

\bibitem{HarcosThorner2}
G.~{Harcos} and J.~{Thorner}.
\newblock {Tatuzawa's theorem for Rankin-Selberg $L$-functions}.
\newblock {\em arXiv e-prints}, page arXiv:2508.10844, Aug. 2025.

\bibitem{HoffsteinLockhart}
J.~Hoffstein and P.~Lockhart.
\newblock Coefficients of {M}aass forms and the {S}iegel zero.
\newblock {\em Ann. of Math. (2)}, 140(1):161--181, 1994.
\newblock Appendix by Dorian Goldfeld, Hoffstein and Daniel Lieman.

\bibitem{HoffsteinRamakrishnan}
J.~Hoffstein and D.~Ramakrishnan.
\newblock Siegel zeros and cusp forms.
\newblock {\em Int. Math. Res. Not.}, (6):279--308, 1995.

\bibitem{Humphries}
P.~Humphries and F.~Brumley.
\newblock Standard zero-free regions for {R}ankin-{S}elberg {$L$}-functions via
  sieve theory.
\newblock {\em Math. Z.}, 292(3-4):1105--1122, 2019.

\bibitem{HumphriesThorner}
P.~Humphries and J.~Thorner.
\newblock Towards a $\mathrm{GL}_n$ variant of the {H}oheisel phenomenon.
\newblock {\em Trans. Amer. Math. Soc.}, 375(3):1801--1824, 2022.

\bibitem{IK}
H.~Iwaniec and E.~Kowalski.
\newblock {\em Analytic number theory}, volume~53 of {\em American Mathematical
  Society Colloquium Publications}.
\newblock American Mathematical Society, Providence, RI, 2004.

\bibitem{JPSS}
H.~Jacquet, I.~I. Piatetskii-Shapiro, and J.~A. Shalika.
\newblock Rankin-{S}elberg convolutions.
\newblock {\em Amer. J. Math.}, 105(2):367--464, 1983.

\bibitem{JS}
H.~Jacquet and J.~A. Shalika.
\newblock A non-vanishing theorem for zeta functions of {$\mathrm{GL}_{n}$}.
\newblock {\em Invent. Math.}, 38(1):1--16, 1976/77.

\bibitem{Kim}
H.~H. Kim.
\newblock Functoriality for the exterior square of {$\mathrm{GL}_4$} and the
  symmetric fourth of {$\mathrm{GL}_2$}.
\newblock {\em J. Amer. Math. Soc.}, 16(1):139--183, 2003.
\newblock With appendix 1 by Dinakar Ramakrishnan and appendix 2 by Kim and
  Peter Sarnak.

\bibitem{KimShahidi2}
H.~H. Kim and F.~Shahidi.
\newblock Cuspidality of symmetric powers with applications.
\newblock {\em Duke Math. J.}, 112(1):177--197, 2002.

\bibitem{KimShahidi}
H.~H. Kim and F.~Shahidi.
\newblock Functorial products for $\mathrm{GL}_2\times\mathrm{GL}_3$ and the
  symmetric cube for $\mathrm{GL}_2$.
\newblock {\em Ann. of Math. (2)}, 155(3):837--893, 2002.
\newblock Appendix by Colin J. Bushnell and Guy Henniart.

\bibitem{Lapid}
E.~Lapid.
\newblock On the {H}arish-{C}handra {S}chwartz space of {$G(F)\backslash
  G(\mathbb{A})$}.
\newblock In {\em Automorphic representations and {$L$}-functions}, volume~22
  of {\em Tata Inst. Fundam. Res. Stud. Math.}, pages 335--377. Tata Inst.
  Fund. Res., Mumbai, 2013.
\newblock Appendix by Farrell Brumley.

\bibitem{Luo}
W.~Luo.
\newblock Non-existence of {S}iegel zeros for cuspidal functorial products on
  {$GL(2) \times GL(3)$}.
\newblock {\em Proc. Amer. Math. Soc.}, 151(5):1915--1919, 2023.

\bibitem{LRS}
W.~Luo, Z.~Rudnick, and P.~Sarnak.
\newblock On the generalized {R}amanujan conjecture for {$\mathrm{GL}(n)$}.
\newblock In {\em Automorphic forms, automorphic representations, and
  arithmetic ({F}ort {W}orth, {TX}, 1996)}, volume~66 of {\em Proc. Sympos.
  Pure Math.}, pages 301--310. Amer. Math. Soc., Providence, RI, 1999.

\bibitem{MoeglinWaldspurger}
C.~M\oe~glin and J.-L. Waldspurger.
\newblock Le spectre r\'{e}siduel de {$\mathrm{GL}(n)$}.
\newblock {\em Ann. Sci. \'{E}cole Norm. Sup. (4)}, 22(4):605--674, 1989.

\bibitem{MullerSpeh}
W.~M\"{u}ller and B.~Speh.
\newblock Absolute convergence of the spectral side of the {A}rthur trace
  formula for $\mathrm{GL}_n$.
\newblock {\em Geom. Funct. Anal.}, 14(1):58--93, 2004.
\newblock Appendix by E. M. Lapid.

\bibitem{NewtonThorne3}
J.~{Newton} and J.~A. {Thorne}.
\newblock {Symmetric power functoriality for Hilbert modular forms}.
\newblock {\em Ann. of Math. (2)}.
\newblock Accepted for publication.

\bibitem{NewtonThorne}
J.~Newton and J.~A. Thorne.
\newblock Symmetric power functoriality for holomorphic modular forms.
\newblock {\em Publ. Math. Inst. Hautes \'{E}tudes Sci.}, 134:1--116, 2021.

\bibitem{NewtonThorne2}
J.~Newton and J.~A. Thorne.
\newblock Symmetric power functoriality for holomorphic modular forms, {II}.
\newblock {\em Publ. Math. Inst. Hautes \'{E}tudes Sci.}, 134:117--152, 2021.

\bibitem{Ramakrishnan}
D.~Ramakrishnan.
\newblock Modularity of the {R}ankin-{S}elberg {$L$}-series, and multiplicity
  one for $\mathrm{SL}(2)$.
\newblock {\em Ann. of Math. (2)}, 152(1):45--111, 2000.

\bibitem{Ramakrishnan_exercise}
D.~Ramakrishnan.
\newblock An exercise concerning the selfdual cusp forms on {${\rm GL}(3)$}.
\newblock {\em Indian J. Pure Appl. Math.}, 45(5):777--785, 2014.

\bibitem{RamakrishnanWang}
D.~Ramakrishnan and S.~Wang.
\newblock On the exceptional zeros of {R}ankin-{S}elberg {$L$}-functions.
\newblock {\em Compositio Math.}, 135(2):211--244, 2003.

\bibitem{RamakrishnanWang2}
D.~Ramakrishnan and S.~Wang.
\newblock A cuspidality criterion for the functorial product on
  $\mathrm{GL}(2)\times\mathrm{GL}(3)$ with a cohomological application.
\newblock {\em Int. Math. Res. Not.}, (27):1355--1394, 2004.

\bibitem{RamakrishnanYang}
D.~Ramakrishnan and L.~Yang.
\newblock A constraint for twist equivalence of cusp forms on
  {$\mathrm{GL}(n)$}.
\newblock {\em Funct. Approx. Comment. Math.}, 65(1):105--117, 2021.

\bibitem{Shahidi}
F.~Shahidi.
\newblock On certain {$L$}-functions.
\newblock {\em Amer. J. Math.}, 103(2):297--355, 1981.

\bibitem{Takeda}
S.~Takeda.
\newblock On a certain metaplectic {E}isenstein series and the twisted
  symmetric square {$L$}-function.
\newblock {\em Math. Z.}, 281(1-2):103--157, 2015.

\bibitem{Thorner_Siegel}
J.~Thorner.
\newblock Exceptional zeros of {R}ankin--{S}elberg {$L$}-functions and joint
  {S}ato--{T}ate distributions.
\newblock {\em Int. Math. Res. Not. IMRN}, (20):rnaf307, 2025.

\bibitem{Wattanawanichkul}
N.~{Wattanawanichkul}.
\newblock {A metric approach to zero-free regions for $L$-functions}.
\newblock {\em arXiv e-prints}, page arXiv:2504.05606, Apr. 2025.

\end{thebibliography}

\end{document}